\documentclass[12pt]{article}

\usepackage{mathrsfs}

\usepackage{amssymb, amsthm, amsfonts, amsxtra, amsmath}
\usepackage{latexsym}
\usepackage{mathabx}

\theoremstyle{change}

\newtheorem{Theorem}{Theorem}[section]
\newtheorem{Def}[Theorem]{Definition}
\newtheorem{Lem}[Theorem]{Lemma}
\newtheorem{Prop}[Theorem]{Proposition}
\newtheorem{Cor}[Theorem]{Corollary}
\newtheorem{Not}[Theorem]{Notation}
\newtheorem{Exa}[Theorem]{Example}
\newtheorem{Def-Prop}[Theorem]{Definition-Proposition}

\date{}

\begin{document}

\hyphenation{Wo-ro-no-wicz}

\title{On projective representations for compact quantum groups}
\author{Kenny De Commer\footnote{Supported in part by the ERC Advanced Grant 227458
OACFT ``Operator Algebras and Conformal Field Theory" }\\ \small Dipartimento di Matematica,  Universit\`{a} degli Studi di Roma Tor Vergata\\
\small Via della Ricerca Scientifica 1, 00133 Roma, Italy\\ \\ \small e-mail: decommer@mat.uniroma2.it}
\maketitle

\newcommand{\acnabla}{\nabla\!\!\!{^\shortmid}}
\newcommand{\undersetmin}[2]{{#1}\underset{\textrm{min}}{\otimes}{#2}}
\newcommand{\otimesud}[2]{\overset{#2}{\underset{#1}{\otimes}}}
\newcommand{\qbin}[2]{\left[ \begin{array}{c} #1 \\ #2 \end{array}\right]_{q^2}}

\newcommand{\qortc}[4]{\,\;_1\varphi_1\left(\begin{array}{c} #1  \\#2 \end{array}\mid #3,#4\right)}
\newcommand{\qortPsi}[4]{\Psi\left(\begin{array}{c} #1  \\#2 \end{array}\mid #3,#4\right)}
\newcommand{\qorta}[5]{\,\;_2\varphi_1\left(\begin{array}{cc} #1 & #2 \\ & \!\!\!\!\!\!\!\!\!\!\!#3 \end{array}\mid #4,#5\right)}
\newcommand{\qortd}[6]{\,\;_2\varphi_2\left(\begin{array}{cc} #1 & #2 \\ #3 & #4 \end{array}\mid #5,#6\right)}
\newcommand{\qortb}[7]{\,\;_3\varphi_2\left(\begin{array}{ccc} #1 & #2 & #3 \\ & \!\!\!\!\!\!\!\!#4 & \!\!\!\!\!\!\!\!#5\end{array}\mid #6,#7\right)}

\newcommand{\otimesmin}{\underset{\textrm{min}}{\otimes}}
\newcommand{\bigback}{\!\!\!\!\!\!\!\!\!\!\!\!\!\!\!\!\!\!\!\!\!\!\!\!}

\abstract{\noindent We study actions of compact quantum groups on type I factors, which may be interpreted as projective representations of compact quantum groups. We generalize to this setting some of Woronowicz' results concerning Peter-Weyl theory for compact quantum groups. The main new phenomenon is that for general compact quantum groups (more precisely, those which are not of Kac type), not all irreducible projective representations have to be finite-dimensional. As applications, we consider the theory of projective representations for the compact quantum groups associated to group von Neumann algebras of discrete groups, and consider a certain non-trivial projective representation for quantum $SU(2)$.}\vspace{0.3cm}


\noindent AMS 2010 \emph{Mathematics subject classification}: 17B37, 81R50, 16T15


\section*{Introduction}

\noindent It is well-known that for compact groups, one can easily extend the main theorems of the Peter-Weyl theory to cover also \emph{projective} representations. In this article, we will see that if one tries to do the same for Woronowicz' compact \emph{quantum} groups, one confronts at least one surprising novelty: not all irreducible projective representations of a compact quantum group have to be finite-dimensional. On the other hand, one will still be able to decompose any projective representation into a direct sum of irreducible ones, and to determine certain orthogonality relations between the matrix coefficients of irreducible projective representations.\\

\noindent The main tool we will use in this article are the \emph{Galois co-objects} which we introduced in \cite{DeC1}. Indeed, we showed there that when one quantizes the notion of a projective representation, this structure plays the role of a `generalized 2-cocycle function'.\\

\noindent In the \emph{first} section of this article, we will develop a structure theory for such Galois co-objects in the setting of compact quantum groups. A lot of the techniques we use are directly inspired by the theory of the compact quantum groups themselves.\\

\noindent In the \emph{second section}, we will show that such Galois co-objects can be dualized into Galois \emph{objects} for their dual discrete quantum groups, a concept which was introduced in \cite{DeC3}.\\

\noindent In the \emph{third section}, which, except for the last part, is independent from the more technical second section, we present a Peter-Weyl theory for projective representations of compact quantum groups. We also show how projective representations give rise to module categories over the tensor category of the (ordinary) representations, and introduce the notion of \emph{fusion rules} between irreducible projective representations and (ordinary) irreducible representations.\\

\noindent In the \emph{fourth section}, we will give some details on the `reflection technique' introduced in \cite{DeC1}. We showed there that from any Galois co-object for a given compact quantum group, one can create a (possibly) new \emph{locally compact} quantum group. We will show that the type of this quantum group (namely whether it is compact or not) is intimately tied up with the behavior of the Galois co-object itself.\\

\noindent In the \emph{fifth section}, we will consider the special case of compact \emph{Kac} algebras. We show that in this case, all irreducible projective representations will be finite-dimensional, and the theory becomes essentially algebraic.\\

\noindent In the \emph{sixth and seventh section}, we further specialize. We first quickly consider the case of \emph{finite quantum groups} (i.e.~ \emph{finite-dimensional} Kac algebras), for which we can mostly refer to the existing literature. Then we will treat \emph{co-commutative} compact Kac algebras, which correspond to group von Neumann algebras of discrete groups. In this case, the projective representations turn out to be classified by certain special 2-cohomology classes of \emph{finite subgroups} of the associated discrete group. In particular, we will be able to deduce that the group von Neumann algebra of a torsionless discrete group admits no non-trivial 2-cocycles. These results will be proven using only the material in the first section and the first part of the third section.\\

\noindent In the \emph{eighth section}, we give a concrete example of what can happen in the non-Kac case by considering a particular non-trivial Galois co-object for the compact quantum group $SU_q(2)$. We compute explicitly all its associated projective representations, provide the corresponding orthogonality relations and calculate the fusion rules.\\

\noindent \qquad \emph{Notations and conventions}\\

\noindent We will assume that
 all our Hilbert spaces are separable, and we take the inner product to be conjugate linear in the second argument. We also assume that all our von Neumann algebras have separable predual.\\

\noindent By $\iota$, we denote the identity map on a set.\\

\noindent We denote by $\odot$ the algebraic tensor product between vector spaces, by $\otimes$ the tensor product between Hilbert spaces, and by $\overline{\otimes}$ the spatial tensor product between von Neumann algebras.\\

\noindent By $\Sigma$ we denote the flip map between a tensor product of Hilbert spaces: \[\Sigma: \mathscr{H}_1\otimes \mathscr{H}_2\rightarrow \mathscr{H}_2\otimes \mathscr{H}_1: \xi_1\otimes \xi_2\rightarrow \xi_2\otimes \xi_1.\]

\noindent When $A\subseteq B(\mathscr{H}_1,\mathscr{H}_2)$ and $B\subseteq B(\mathscr{H}_2,\mathscr{H}_3)$ are linear spaces of maps between certain Hilbert spaces, we will denote $B\cdot A = \{\sum_{i=1}^n b_ia_i \mid n\in \mathbb{N}_0,b_i\in B,a_i\in A\}$.\\

\noindent We use the \emph{leg numbering notation} for operators on tensor products of Hilbert spaces. E.g., if $Z: \mathscr{H}^{\otimes 2}\rightarrow \mathscr{H}^{\otimes 2}$ is a certain operator, we denote by $Z_{13}$ the operator $\mathscr{H}^{\otimes 3}\rightarrow \mathscr{H}^{\otimes 3}$ acting as $Z$ on the first and third factor, and as the identity on the second factor.\\

\noindent At certain points, we will need the theory of weights on von Neumann algebras, which is treated in detail in the first chapters of \cite{Tak1}. When $M$ is a von Neumann algebra, and $\varphi: M^+\rightarrow \lbrack 0,\infty\rbrack$ is a normal semi-finite faithful (nsf) weight on $M$, we denote \[\mathscr{N}_{\varphi}=\{x\in M\mid \varphi(x^*x)<\infty\}\] for the space of square integrable elements, $\mathscr{M}_{\varphi}^+$ for the space of positive integrable elements, and $\mathscr{M}_{\varphi}$ for the linear span of $\mathscr{M}_{\varphi}^+$.\\

\section{Galois co-objects for compact quantum groups}

\noindent We begin with introducing the following concepts.

\begin{Def}\label{Defvnaqg}\label{DefCQG} A \emph{von Neumann bialgebra} $(M,\Delta_M)$ consists of a von Neumann algebra $M$ and a faithful normal unital $^*$-homomorphism $\Delta_M: M\rightarrow M\bar{\otimes} M$, satisfying the coassociativity condition \[(\Delta_M\otimes \iota)\Delta_M = (\iota\otimes \Delta_M)\Delta_M.\]

\noindent A von Neumann bialgebra $(M,\Delta_M)$ is called a \emph{compact Woronowicz algebra} (\cite{Mas1},\cite{Kus2}) if there exists a normal state $\varphi_M$ on $M$ which is \emph{$\Delta_M$-invariant}: \[(\varphi_M\otimes \iota)\Delta_M(x) = (\iota\otimes \varphi_M)\Delta_M(x) = \varphi_M(x)1\qquad \textrm{for all }x\in M.\]

\noindent A compact Woronowicz algebra is called a \emph{compact Kac algebra} if there exists a normal $\Delta_M$-invariant \emph{tracial} state $\tau_M$ on $M$.

\end{Def}

\noindent\emph{Remarks:} \begin{enumerate} \item von Neumann bialgebras are also referred to as Hopf-von Neumann algebras in the literature. However, we prefer the above terminology, as for example a finite-dimensional Hopf-von Neumann algebra is not necessarily a Hopf algebra. Admittedly, a finite-dimensional von Neumann bialgebra is also not necessarily a bialgebra, as there could be no co-unit, but this seems a lesser ambiguity.
\item It is easy to see that a normal $\Delta_M$-invariant state on a von Neumann bialgebra $(M,\Delta_M)$, when it exists, is unique. One can moreover show that this state will automatically be faithful. We will then always use the notation $\varphi_M$ for it in the general setting, but use the notation $\tau_M$ in the setting of compact Kac algebras to emphasize the traciality.
\item Compact Woronowicz algebras can be characterized as those von Neumann bialgebras arising from Woronowicz' compact quantum groups in the C$^*$-algebra setting (\cite{Wor1},\cite{Wor2}), by performing a GNS-type construction. However, we have decided to focus only on the von Neumann algebraic picture in this paper.
\end{enumerate}

\noindent Let us also introduce the following notations, which will be constantly used in the following.

\begin{Not} Let $(M,\Delta_M)$ be a compact Woronowicz algebra. We denote by $(\mathscr{L}^2(M),\,\,\,,\Lambda_{M})$ the GNS-construction of $M$ with respect to $\varphi_M$. That is, $\mathscr{L}^2(M)$ is the completion of $M$, considered as a pre-Hilbert space with respect to the inner product structure \[\langle x,y\rangle = \varphi_M(y^*x),\] and $\Lambda_M$ is the natural inclusion $M\hookrightarrow \mathscr{L}^2(M)$. We then identify $M$ as a von Neumann subalgebra of $B(\mathscr{L}^2(M))$ by letting $x\in M$ corresponding to the (bounded) closure of the operator \[ \Lambda_M(M) \rightarrow \mathscr{L}^2(M): \Lambda_M(y)\rightarrow \Lambda_M(xy),\qquad \textrm{for all }y\in M.\]

\noindent We will further denote by $\xi_M$ the cyclic and separating vector $\Lambda_M(1_M)$ in $\mathscr{L}^2(M)$, so that $x\xi_{M} = \Lambda_M(x)$ for all $x\in M$.\end{Not}

\noindent The following two unitaries are of fundamental importance.

\begin{Def} Let $(M,\Delta_M)$ be a compact Woronowicz algebra.\\

\noindent The \emph{right regular corepresentation} of $(M,\Delta_M)$ is defined to be the unitary $V \in B(\mathscr{L}^2(M))\bar{\otimes} M$ which is uniquely determined by the formula \[V \Lambda_M(x)\otimes \eta = \Delta_{M}(x) \xi_{M}\otimes \eta,\qquad \textrm{for all }x\in M, \eta\in \mathscr{L}^2(M).\]

\noindent The \emph{left regular corepresentation} of $(M,\Delta_M)$ is defined to be the unitary $W\in M\bar{\otimes} B(\mathscr{L}^2(M))$ which is uniquely determined by the fact that \[W^*\eta\otimes \Lambda_M(x) = \Delta_M(x)\eta\otimes \xi_{M},\qquad \textrm{for all }x\in M,\eta\in \mathscr{L}^2(M).\]

\end{Def}

\noindent We will in the following always use the above notations for these corepresentations. Note that establishing the unitarity of these maps requires some non-trivial work! An approach to compact quantum groups based on the properties of such unitaries can be found in \cite{Baa1}, section 4.\\

\noindent Let us now introduce the notion of a Galois co-object for a compact Woronowicz algebra (see \cite{DeC1}).

\begin{Def}\label{DefGalCo} Let $(M,\Delta_M)$ be a compact Woronowicz algebra. A right \emph{Galois co-object} for $(M,\Delta_M)$ consists of a Hilbert space $\mathscr{L}^2(N)$, a $\sigma$-weakly closed linear space $N\subseteq B(\mathscr{L}^2(M),\mathscr{L}^2(N))$ and a normal linear map $\Delta_N: N\rightarrow N\bar{\otimes}N$, such that the following properties hold: with $N^{\textrm{op}}$ denoting the set \[N^{\textrm{op}} := \{x^*\mid x\in N\} \subseteq B(\mathscr{L}^2(N),\mathscr{L}^2(M)),\] we should have
\begin{enumerate}
\item $N \cdot \mathscr{L}^2(M)$ is norm-dense in $\mathscr{L}^2(N)$, and $N^{\textrm{op}}\cdot \mathscr{L}^2(N)$ is norm-dense in $\mathscr{L}^2(M)$,
\item the space $N$ is a right $M$-module (by composition of operators),
\item for each $x,y\in N$, we have $x^*y \in M$,
\item $\Delta_N(xy)= \Delta_N(x)\Delta_M(y)$ for all $x\in N$ and $y\in M$,
\item $\Delta_N(x)^*\Delta_N(y) = \Delta_M(x^*y)$ for all $x,y\in N$,
\item $\Delta_N$ is coassociative: $(\Delta_N\otimes \iota)\Delta_N = (\iota\otimes \Delta_N)\Delta_N$, and
\item the linear span of the set $\{\Delta_N(x)(y\otimes z)\mid x\in N,y,z\in M\}$ is $\sigma$-weakly dense in $N\bar{\otimes}N$.\end{enumerate}

\noindent If $(N_1,\Delta_{N_1})$ and $(N_2,\Delta_{N_2})$ are two Galois co-objects for a von Neumann bialgebra $(M,\Delta_M)$, we call them \emph{isomorphic} if there exists a unitary $u:\mathscr{L}^2(N_1)\rightarrow \mathscr{L}^2(N_2)$ such that $uN_1 = N_2$ and \[\Delta_{N_2}(ux) = (u\otimes u)\Delta_{N_1}(x) \qquad \textrm{for all }x\in N_1.\]

\end{Def}

\noindent \emph{Remarks:} \begin{enumerate}\item The previous definition can be shown to be equivalent with the one presented in \cite{DeC1}, Definition 0.5. Also remark that the previous conditions can be grouped together as follows: a Galois co-object is a right \emph{Morita (or imprimitivity) Hilbert $M$-module} (conditions 1 to 3) with a $\Delta_M$-compatible coalgebra structure (conditions 4 and 5 and condition 6) which is `non-degenerate' (condition 7).
\item A trivial example of a right $(M,\Delta_M)$-Galois co-object is $(M,\Delta_M)$ itself. Indeed, the final condition even holds in a stronger form, as it can be shown that already $\{\Delta_M(x)(1\otimes y)\mid x,y\in M\}$ is $\sigma$-weakly dense in $M\bar{\otimes}M$ for compact Woronowicz algebras. It follows that this stronger condition is then in fact true for \emph{all} Galois co-objects for compact Woronowicz algebras.
\item A treatment of Galois co-objects in the setting of Hopf algebras can be found in \cite{Cae1}.
\end{enumerate}

\noindent One can similarly define the notion of a \emph{left} Galois co-object. Left Galois co-objects can be created from right ones in the following way.

\begin{Def}\label{DefCoop} Let $(N,\Delta_N)$ be a right Galois co-object for the compact Woronowicz algebra $(M,\Delta_M)$.\\

\noindent We call the couple $(N^{\textrm{op}},\Delta_{N^{\textrm{op}}})$, consisting of \[N^{\textrm{op}} = \{x^*\mid x\in N\}\subseteq B(\mathscr{L}^2(N),\mathscr{L}^2(M)),\] together with the coproduct \[\Delta_{N^{\textrm{op}}}(x) := \Delta_N(x^*)^*, \qquad x\in N^{\textrm{op}},\] the \emph{opposite (left) Galois co-object} of $(N,\Delta_N)$. It is a left Galois co-object for the compact Woronowicz algebra $(M,\Delta_M)$.\\

\noindent We call the couple $(N^{\textrm{cop}},\Delta_{N^{\textrm{cop}}})$, where \[N^{\textrm{cop}} = N \subseteq B(\mathscr{L}^2(M),\mathscr{L}^2(N))\] and \[\Delta_{N^{\textrm{cop}}}=\Delta_N^{\textrm{op}}:N\rightarrow N\bar{\otimes}N: x\rightarrow \Sigma \Delta_N(x)\Sigma,\qquad \textrm{for all }x\in N,\] the \emph{co-opposite (right) Galois co-object} of $(N,\Delta_N)$. It is a right Galois co-object for the compact Woronowicz algebra $(M,\Delta_M^{op})$.
\end{Def}

\noindent The following notation will be useful.

\begin{Not} Let $(M,\Delta_M)$ be a compact Woronowicz algebra, and $(N,\Delta_N)$ a right Galois co-object for $(M,\Delta_M)$. We denote \[\Lambda_N: N\rightarrow \mathscr{L}^2(N): x\rightarrow x\xi_{M}.\]
\end{Not}

\noindent \emph{Remark:} By the second condition in Definition \ref{DefGalCo}, we know that $N$ is a right $M$-module, and then we trivially have that \[\Lambda_N(xy) = x\Lambda_M(y), \qquad \textrm{for all }x\in N,y\in M.\] By the third condition in that definition, together with the faithfulness of $\varphi_M$, we see that $\Lambda_N$ is injective and that \[\langle \Lambda_N(x),\Lambda_N(y)\rangle = \varphi_M(y^*x),\qquad \textrm{for all }x,y\in N.\] And finally, by the first (and second) condition in that definition, we see that $\Lambda_N$ has norm-dense range.\\

\noindent One can construct for a Galois co-object $(N,\Delta_N)$ certain unitaries which are analogous to the regular corepresentations for a compact Woronowicz algebra (and coincide with them in case $(N,\Delta_N)=(M,\Delta_M)$).

\begin{Prop}\label{PropElProp} Let $(M,\Delta_M)$ be a compact Woronowicz algebra, and $(N,\Delta_N)$ a right Galois co-object for $(M,\Delta_M)$.

\begin{enumerate}\item There exists a unitary \[\widetilde{V}: \mathscr{L}^2(N) \otimes \mathscr{L}^2(M) \rightarrow \mathscr{L}^2(N)\otimes \mathscr{L}^2(N)\] which is uniquely determined by the property that for all $\eta \in \mathscr{L}^2(M)$ and $x\in N$, we have \[ \widetilde{V}\, \Lambda_N(x) \otimes \eta = \Delta_N(x) \xi_{M}\otimes \eta.\] Similarly, there exists a unitary \[\widetilde{W}: \mathscr{L}^2(N)\otimes \mathscr{L}^2(N)\rightarrow \mathscr{L}^2(M) \otimes \mathscr{L}^2(N) ,\] uniquely determined by the property that for all $\eta\in \mathscr{L}^2(M)$ and $x\in N$, we have \[\widetilde{W}^*\, \eta\otimes \Lambda_N(x) = \Delta_N(x) \eta\otimes \xi_{M}.\]
\item We have $\widetilde{V} \in B(\mathscr{L}^2(N))\bar{\otimes} N$ and $\widetilde{W}^* \in N\bar{\otimes} B(\mathscr{L}^2(N))$.
\item For $x\in N$, we have \[\Delta_N(x) =  \widetilde{V}(x\otimes 1)V^*=\widetilde{W}^*(1\otimes x)W.\]
\item The following `pentagonal identities' hold: \[\widetilde{V}_{12}\widetilde{V}_{13}V_{23} = \widetilde{V}_{23}\widetilde{V}_{12}\] as maps from $\mathscr{L}^2(N)\otimes \mathscr{L}^2(M)\otimes \mathscr{L}^2(M)$ to $\mathscr{L}^2(N)\otimes \mathscr{L}^2(N)\otimes \mathscr{L}^2(N)$, and  \[W_{12}\widetilde{W}_{13}\widetilde{W}_{23} = \widetilde{W}_{23}\widetilde{W}_{12}\]
    as maps from $\mathscr{L}^2(N)\otimes \mathscr{L}^2(N)\otimes \mathscr{L}^2(N)$ to $\mathscr{L}^2(M)\otimes \mathscr{L}^2(M)\otimes \mathscr{L}^2(N)$.
\item The following identities hold: \[(\iota\otimes \Delta_N)\widetilde{V} =\widetilde{V}_{12}\widetilde{V}_{13},\]\[(\Delta_N\otimes \iota)(\widetilde{W}^*) = \widetilde{W}_{23}^*\widetilde{W}_{13}^*.\]
\end{enumerate}

\end{Prop}

\begin{proof} The statements for $\widetilde{W}$ follow immediately from the ones for $\widetilde{V}$, by considering the co-opposite Galois co-object.\\

\noindent We then refer to \cite{DeC1} for the proofs of the first four statements (Proposition 2.3 for the first and second assertion, Proposition 2.4 for the third and fourth). The fifth statement follows immediately from combining the three preceding ones.
\end{proof}

\noindent \emph{Remark:} Although we referred to \cite{DeC1}, we want to stress that these assertions are quite straightforward to prove. For example, the surjectivity of $\widetilde{V}$ follows quite immediately from the seventh condition in Definition \ref{DefGalCo}, combined with the surjectivity of $V$.\\

\begin{Def} We call the unitary $\widetilde{V}$ appearing in the previous proposition the \emph{right regular $(N,\Delta_N)$-corepresentation} of $(M,\Delta_M)$. We call the unitary $\widetilde{W}$ the \emph{left regular $(N^{\textrm{op}},\Delta_{N^{\textrm{op}}})$-corepresentation} of $(M,\Delta_M)$ (where we recall that $(N^{\textrm{op}},\Delta_{N^{\textrm{op}}})$ is the left Galois co-object opposite to $(N,\Delta_N)$, see Definition \ref{DefCoop}).
\end{Def}

\noindent\emph{Remark:} The general notion of an `$(N,\Delta_N)$-corepresentation' will be introduced in the third section.\\

\noindent \emph{For the rest of this section, we will fix a compact Woronowicz algebra $(M,\Delta_M)$ and a right Galois co-object $(N,\Delta_N)$ for $(M,\Delta_M)$}. We then further keep denoting by $V$ and $W$ the right and left regular corepresentations of $(M,\Delta_M)$, and by $\widetilde{V}$ and $\widetilde{W}$ the right regular $(N,\Delta_N)$- and left regular $(N^{\textrm{op}},\Delta_{N^{\textrm{op}}})$-corepresentation of $(M,\Delta_M)$.\\

\noindent Our following lemma improves the second assertion in Proposition \ref{PropElProp}.

\begin{Lem}\label{LemDen1} The following equalities hold: \begin{eqnarray*} N &=&\{(\omega\otimes \iota)(\widetilde{V}) \mid \omega\in B(\mathscr{L}^2(N))_*\}^{\sigma\textrm{-weak closure}} \\ &=& \{(\iota\otimes\omega)(\widetilde{W}^*)\mid \omega\in B(\mathscr{L}^2(N))_*\}^{\sigma\textrm{-weak closure}}.\end{eqnarray*}\end{Lem}

\begin{proof} We will again only prove the first identity, as the second one then follows by symmetry.\\

\noindent For $\xi,\eta\in \mathscr{L}^2(N)$, denote by $\omega_{\xi,\eta}$ the normal functional on $B(\mathscr{L}^2(N))$ determined by $\omega_{\xi,\eta}(x) = \langle x\xi,\eta\rangle $ for $x\in B(\mathscr{L}^2(N))$. Then for $x,y\in N$, a straightforward computation shows that \[(\omega_{\Lambda_N(x),\Lambda_N(y)} \otimes \iota)(\widetilde{V})=(\varphi_{M}\otimes \iota)((y^*\otimes 1)\Delta_N(x)).\] It is thus enough to prove that the linear span of such elements is $\sigma$-weakly dense in $N$.\\

\noindent Suppose that this were not so. Then we could find a non-zero $\omega\in N_*$ such that \[\varphi_M(y^* (\iota\otimes \omega)(\Delta_N(x))) = 0 \qquad\textrm{for all }x,y\in N.\] Taking $y$ equal to $(\iota\otimes \omega)(\Delta_N(x))$, we would have $(\iota\otimes \omega)(\Delta_N(x)) =0$ for all $x\in N$ by faithfulness of $\varphi_{M}$. But then also \[(\iota\otimes \omega)(\Delta_N(x)(m\otimes 1)) = 0 \qquad \textrm{for all }x\in N,m\in M.\] Now the set $\{\Delta_M(m_1)(m_2\otimes 1)\mid m_1,m_2\in M\}$ has $\sigma$-weakly dense linear span in $M\bar{\otimes} M$. It then follows, by the conditions 2, 4 and 7 in Definition \ref{DefGalCo} that \[(\iota\otimes \omega)(z)=0 \qquad \textrm{for all }z\in N\bar{\otimes}N,\] and so necessarily $\omega=0$, a contradiction.

\end{proof}

\noindent The following result will allow us to obtain a decomposition for $\widetilde{W}$ and $\widetilde{V}$.

\begin{Prop}\label{PropDen2} Denote by $\widehat{N}\subseteq B(\mathscr{L}^2(N))$ the von Neumann algebra \[\widehat{N} = \{x\in B(\mathscr{L}^2(N))\mid \widetilde{V}^*(x\otimes 1)\widetilde{V} = x\otimes 1\}.\]

\noindent Then $\widehat{N}$ satisfies the following properties.

\begin{enumerate}\item The von Neumann algebra $\widehat{N}$ is an $l^{\infty}$-sum of type $I$-factors.
\item The equality $\widehat{N} = \{(\omega\otimes \iota)(\widetilde{W}^*)\mid \omega\in B(\mathscr{L}^2(M),\mathscr{L}^2(N))_*\}^{\sigma\textrm{-weak closure}}$ holds.
\end{enumerate}

\end{Prop}

\noindent \emph{Remark:} In the special case where $(N,\Delta_N)$ equals $(M,\Delta_M)$ considered as a right Galois co-object over itself, one denotes the above von Neumann algebra as $\widehat{M}$.

\begin{proof} Consider the unital normal faithful $^*$-homomorphism \[ \textrm{Ad}_L: B(\mathscr{L}^2(N))\rightarrow M\bar{\otimes} B(\mathscr{L}^2(N)): x\rightarrow \Sigma \widetilde{V}^*(x\otimes 1)\widetilde{V}\Sigma.\] Then by Proposition \ref{PropElProp}.5, it follows that $\textrm{Ad}_L$ is a coaction by $(M,\Delta_M)$: \[(\Delta_M\otimes \iota)\textrm{Ad}_L = (\iota\otimes \textrm{Ad}_L)\textrm{Ad}_L.\]

\noindent Hence $\widehat{N}$ is precisely the set $B(\mathscr{L}^{2}(N))^{\textrm{Ad}_L}$ of $\textrm{Ad}_L$-fixed elements in $B(\mathscr{L}^2(N))$, that is, the set of elements satisfying $\textrm{Ad}_L(x) = 1\otimes x$. It is well-known (and easy to see) that the map \[E: B(\mathscr{L}^2(N)) \rightarrow B(\mathscr{L}^2(N)): x\rightarrow (\varphi_{M}\otimes \iota)\textrm{Ad}_L(x)\] is then a \emph{normal} conditional expectation of $B(\mathscr{L}^2(N))$ onto $\widehat{N}$. This forces $\widehat{N}$ to be an $l^{\infty}$-direct sum of type $I$-factors (see for example Exercise IX.4.1 in \cite{Tak1}).\\

\noindent We now prove the second point. First of all, remark that \[\widetilde{V}_{23}\widetilde{W}_{12}^*=\widetilde{W}_{12}^*\widetilde{V}_{23},\] which follows from a straightforward computation. From this, it is easy to get that \[(\iota\otimes \textrm{Ad}_L)(\widetilde{W}) = \widetilde{W}_{13},\] and so all elements of the form $(\omega\otimes \iota)(\widetilde{W})$ with $\omega\in B(\mathscr{L}^2(N),\mathscr{L}^2(M))_*$ lie in $\widehat{N}$. We next show that all elements of $\widehat{N}$ can be approximated $\sigma$-weakly by such elements.\\

\noindent For $\omega_1,\omega_2\in N_*$, denote \[\omega_1*\omega_2 := (\omega_1\otimes \omega_2)\circ \Delta_N\,\in N_*.\] For $\xi,\eta\in B(\mathscr{L}^2(N))$, denote $\theta_{\xi,\eta}$ for the rank one operator $\zeta\rightarrow \langle \zeta,\eta\rangle \xi$ on $\mathscr{L}^2(N)$, and denote $\omega_{\xi,\eta}$ for the normal functional $x\rightarrow \langle x\xi,\eta\rangle$. Choose $b,x,y\in N$, and denote \[a= (\omega_{\Lambda_N(x),\Lambda_N(y)}\otimes \iota)(\widetilde{V}) \in N,\qquad S_N(a) = (\omega_{\Lambda_N(x),\Lambda_N(y)}\otimes \iota)(\widetilde{V}^*) \in N^{\textrm{op}},\] where we recall that $N^{\textrm{op}} = \{x^*\mid x\in N\} \subseteq B(\mathscr{L}^2(N),\mathscr{L}^2(M))$. We will prove the identity \begin{equation}\label{EqCond} E (\theta_{\Lambda_N(a),\Lambda_N(b)}) = ((\varphi_{M}(b^*\,\cdot\,)* \varphi_{M}(S_N(a)\,\cdot\,))\otimes \iota)(\widetilde{W}^*),\end{equation} where $E$ is the conditional expectation defined in the first part of the proof, and where $\varphi_M(b^*\,\cdot\,)$ and $\varphi_M(S_N(a)\,\cdot\,)$ are the obvious normal functionals on $N$. As the linear span of the $\theta_{\Lambda_N(a),\Lambda_N(b)}$ is $\sigma$-weakly dense in $B(\mathscr{L}^2(N))$ by Lemma \ref{LemDen1}, and as $E$ is a normal map with $\widehat{N}$ as its range, the second point of the proposition will follow from this identity.\\

\noindent To prove the identity (\ref{EqCond}), choose further $c,d\in N$. It is sufficient to prove then that \begin{equation}\label{EqCond2} \langle E (\theta_{\Lambda_N(a),\Lambda_N(b)})\cdot \Lambda_N(c),\Lambda_N(d)\rangle  = \langle ((\varphi_{M}(b^*\,\cdot\,)* \varphi_{M}(S_N(a)\,\cdot\,))\otimes \iota)(\widetilde{W}^*)\cdot \Lambda_N(c),\Lambda_N(d)\rangle.\end{equation} We remark now that $a$ and $S_N(a)$ can also be rewritten in the following form, by a simple computation involving only the definition of $\widetilde{V}$: \[ a = (\varphi_N\otimes \iota)((y^*\otimes 1)\Delta_N(x)),\qquad S_N(a) = (\varphi_M\otimes \iota)(\Delta_N(y)^*(x\otimes 1)).\] Using again the definition of $\widetilde{V}$, the left hand of equation (\ref{EqCond2}) then simplifies to \begin{equation}\label{EqCond3}(\varphi_M\otimes \varphi_M\otimes \varphi_M\otimes \varphi_M)((y^*\otimes 1\otimes b^*\otimes 1)\Delta_N(d)_{24}^*\Delta_N(c)_{34}\Delta_N(x)_{12}).\end{equation} On the other hand, using the definition of $\widetilde{W}^*$, we get that the right hand side of equation (\ref{EqCond2}) becomes \begin{equation}\label{EqCond4}(\varphi_M\otimes\varphi_M\otimes \varphi_M\otimes \varphi_M)((1\otimes b^*\otimes 1\otimes d^*)\Delta_N(y)_{13}^*(x\otimes \Delta_N^{(2)}(c))),\end{equation} where $\Delta_N^{(2)}(c) = (\iota\otimes \Delta_N)\Delta_N(c)$.
In both expressions (\ref{EqCond3}) and (\ref{EqCond4}), we can write \[v=(\varphi_M\otimes \iota)((b^*\otimes 1)\Delta_N(c)),\] and we then have to prove that \begin{equation}\label{EqCond5} (\varphi_M\otimes \varphi_M\otimes \varphi_M)((y\otimes \Delta_N(d))^*(\Delta_N(x)\otimes v)) = (\varphi_M\otimes \varphi_M\otimes \varphi_M)((\Delta_N(y)\otimes d)^*(x\otimes \Delta_N(v))).\end{equation} Now by the final condition in Definition \ref{DefGalCo} (and the second remark following it), it is enough to show that these two expression are equal when we replace $x\otimes v$ by $\Delta_N(z)(m\otimes 1)$ and $y\otimes d$ by $\Delta_N(w)(n\otimes 1)$, where $w,z\in N$ and $m,n\in M$. But then the left hand side of (\ref{EqCond5}) becomes \[(\varphi_M\otimes \varphi_M\otimes \varphi_M)((n^*\otimes 1\otimes 1)\Delta_M^{(2)}(w^*z)(\Delta_M(m)\otimes 1)),\] which by invariance of $\varphi_M$ collapses to $\varphi_M(n^*w^*zm)$. A similar computation shows that with this replacement, also the right hand side expression in (\ref{EqCond5}) collapses to $\varphi_M(n^*w^*zm)$. This concludes the proof.\end{proof}

\noindent Of course, we then also have \[\widehat{N} = \{(\omega\otimes \iota)(\widetilde{W})\mid \omega\in B(\mathscr{L}^2(N),\mathscr{L}^2(M))_*\}^{\sigma\textrm{-weak closure}},\] which follows immediately by applying the $^*$-operation to both sides of the identity in the second point of the previous proposition.

\begin{Not}\label{NotDecomp} By Proposition \ref{PropDen2}.1, we may identify the center $\mathscr{Z}(\widehat{N})$ of $\widehat{N}$ with $l^{\infty}(I_N)$, for some countable set $I_N$. Denoting $p_r$ the minimal central projection in $\mathscr{Z}(\widehat{N})$ associated to the element $r\in I_N$, we may further identify $p_r\widehat{N}$ with $B(\mathscr{H}_r)$ for some separable Hilbert space $\mathscr{H}_r$. We also denote \[n_r := \textrm{dim}(\mathscr{H}_r) \in \mathbb{N}_0\cup \{\infty\}.\]\end{Not}

\begin{Prop} The unital normal faithful $^*$-homomorphism \[\textrm{Ad}_R: \widehat{N}\rightarrow B(\mathscr{L}^2(N))\bar{\otimes} M: x\rightarrow \Sigma \widetilde{W}(1\otimes x)\widetilde{W}^*\Sigma\] restricts to a $^*$-homomorphism $\widehat{N}\rightarrow \widehat{N}\bar{\otimes} M$, and defines in this way a right coaction of $(M,\Delta_M)$ on $\widehat{N}$.\\

\noindent Moreover, the set of fixed elements for $\textrm{Ad}_R$ coincides with the center $\mathscr{Z}(\widehat{N})$ of $\widehat{N}$.
\end{Prop}

\begin{proof} From Proposition \ref{PropDen2}.2, it follows that $\widetilde{W}\in N^{\textrm{op}}\bar{\otimes} \widehat{N}$ (we may apply the weak slice map property as $N^{\textrm{op}}$ is a corner of a von Neumann algebra). Hence $\textrm{Ad}_R(x)\in \widehat{N}\bar{\otimes}M$ for $x\in \widehat{N}$. By applying Proposition \ref{PropElProp}.5, we get \[(\textrm{Ad}_R\otimes \iota)\textrm{Ad}_R(x) = (\iota\otimes \Delta_M)\textrm{Ad}_R(x).\] Hence the first part of the proposition follows.\\

\noindent If further $x\in \widehat{N}$ is a fixed element for $\textrm{Ad}_R$, then it follows that $(1\otimes x)\widetilde{W} = \widetilde{W}(1\otimes x)$. Again by Proposition \ref{PropDen2}.2, we deduce that $xy = yx$ for all $y\in \widehat{N}$, i.e.~ $x\in \mathscr{Z}(\widehat{N})$.
\end{proof}

\begin{Cor}\label{CorAd} Using Notation \ref{NotDecomp} and the notation from the previous proposition, the coaction $\textrm{Ad}_R$ restricts to an ergodic coaction \[\textrm{Ad}_R^{(r)}:B(\mathscr{H}_r) \rightarrow B(\mathscr{H}_r)\bar{\otimes}M\] for each $r\in I_N$.

\end{Cor}

\noindent We recall that a coaction $\alpha$ is called ergodic if the only elements satisfying $\alpha(x)=x\otimes 1$ are scalar multiples of the unit element.\\

\begin{proof} Clearly, as $\mathscr{Z}(\widehat{N})$ consists of the fixed points of $\textrm{Ad}_R$ by the previous proposition, it is immediate that $\textrm{Ad}_R$ indeed restricts to $B(\mathscr{H}_r)$. If then $x$ is a fixed element for $\textrm{Ad}_R^{(r)}$, we have, again by the previous proposition, that $x\in \mathscr{Z}(\widehat{N})\cap B(\mathscr{H}_r)$, and $x$ is a scalar operator.

\end{proof}

\noindent Now as each $\textrm{Ad}_R^{(r)}$ appearing in the previous corollary is ergodic, there exists a \emph{unique} $\textrm{Ad}^{(r)}_R$-invariant state $\phi_{N,r}$ on $B(\mathscr{H}_r)$, determined by the formula \[\phi_{N,r}(x)1_{B(\mathscr{H}_r)} = (\iota\otimes  \varphi_M)(\textrm{Ad}^{(r)}_R(x)),\qquad \textrm{for all }x\in B(\mathscr{H}_r).\]

\begin{Not}\label{NotWeight} If $T_r$ is the positive trace class operator associated to the state $\phi_{N,r}$ on $B(\mathscr{H}_r)$ introduced above, we denote by $T_{r,0}\geq T_{r,1} \geq \ldots$ the descending sequence of its eigenvalues, counting multiplicities. We further fix in $\mathscr{H}_r$ a basis $e_{r,i}$, with $0\leq i< n_r$, such that $e_{r,i}$ is an eigenvector for $T_r$ with eigenvalue $T_{r,i}$.\\

\noindent We denote by $e_{r,ij} \in \widehat{N}$ the matrix units associated to the basis $e_{r,i}$, and we denote by $\omega_{r,ij}$ the following normal functionals on $\widehat{N} \subseteq B(\oplus_{r\in I_N} \mathscr{H}_r)$: \[\omega_{r,ij}(x) = \langle xe_{r,i},e_{r,j}\rangle,\qquad x\in \widehat{N}.\]

\noindent In the special case where $(N,\Delta_N)$ equals $(M,\Delta_M)$ considered as a right Galois co-object over itself, we will denote the $n_r$ as $m_r$, the $T_{r,j}$ as $D_{r,j}$ and the $\mathscr{H}_r$ as $\mathscr{K}_r$, but otherwise keep all notation as above.
\end{Not}

\begin{Theorem}\label{PropOrt} Denote \[\widetilde{W}_{r,ij} = (\iota\otimes \omega_{r,ji})(\widetilde{W}) \in N^{\textrm{op}}\subseteq B(\mathscr{L}^2(N),\mathscr{L}^2(M)).\] Then the following statements hold.
\begin{enumerate} \item The unitary $\widetilde{W}$ equals the strong$^*$ convergent sum $\sum_{r\in I_N}\sum_{i,j=0}^{n_r-1} \widetilde{W}_{r,ij}\otimes e_{r,ij}$.
\item For each $r\in I_N$ and $0\leq i,j< n_r$, we have \[\sum_{k=0}^{n_r-1} \widetilde{W}_{r,ik}\cdot \widetilde{W}_{r,jk}^*  = \delta_{i,j} 1_{\mathscr{L}^2(M)},\] \[\sum_{k=0}^{n_r-1} \widetilde{W}_{r,ki}^*\cdot\widetilde{W}_{r,kj}= \delta_{i,j} 1_{\mathscr{L}^2(N)},\] both sums converging strongly.
\item For each $r\in I_N$ and $0\leq i,j< n_r$, we have \[\Delta_{N}(\widetilde{W}_{r,ij}^*) = \sum_{k=0}^{n_r-1} \widetilde{W}_{r,ik}^*\otimes \widetilde{W}_{r,kj}^*,\] the sum again being a strongly$^*$ converging one.
\item The following orthogonality relations hold: \[\varphi_M(\widetilde{W}_{r,ij}\cdot\widetilde{W}_{s,kl}^*) = \delta_{r,s}\delta_{i,k}\delta_{j,l}\, T_{r,j}, \qquad \textrm{for all }r,s\in I_N, 0\leq i,j<n_r, 0\leq k,l<n_s.\]
\end{enumerate}
\end{Theorem}

\begin{proof} The first point is immediate, and also the second one follows straightforwardly from the unitarity of $\widetilde{W}$. The third point follows from the identity $(\Delta_{N}\otimes \iota)(\widetilde{W}^*) = \widetilde{W}_{23}^*\widetilde{W}_{13}^*$ in Proposition \ref{PropElProp}.5. In the fourth point, the orthogonality relations for $r=s$ follow from writing out the identity \[(\iota\otimes \varphi_M)(\textrm{Ad}_R^{(r)}(e_{r,ij})) = \phi_{N,r}(e_{r,ij}) = \delta_{i,j} T_{r,j}.\]

\noindent Thus the only thing left to show is that $\varphi_M(\widetilde{W}_{r,ij}\cdot\widetilde{W}_{s,kl}^*)=0$ for $r\neq s$. But also here, we can use a standard technique (see e.g. \cite{Wor1}). For suppose that this were not so, and choose $r\neq s$ which violate this condition. Consider, for $x\in B(\mathscr{H}_s,\mathscr{H}_r)$, the element \[F(x) = (\varphi_M\otimes \iota)(\widetilde{W}_r(1\otimes x)\widetilde{W}_s^*) \in B(\mathscr{H}_s,\mathscr{H}_r),\] where of course \[\widetilde{W}_r = (1\otimes p_r)\widetilde{W} = \sum_{i,j=0}^{n_r-1} \widetilde{W}_{r,ij}\otimes e_{r,ij} \in N^{\textrm{op}}\bar{\otimes}B(\mathscr{H}_r).\] By assumption, there must exist an $x$ with $F(x)\neq 0$. Fixing such an $x$, denote $y=F(x)$. Then it is easy to see that \[\widetilde{W}_r(1\otimes y)\widetilde{W}_s^* = (1\otimes y),\] using Proposition \ref{PropElProp}.5 and the $\Delta_M$-invariance of $\varphi_M$. This implies that $y^*y$, resp.~ $yy^*$, is a fixed element for $\textrm{Ad}_R^{(s)}$, resp.~ $\textrm{Ad}_R^{(r)}$. Since these coactions are ergodic, $y^*y$ and $yy^*$ must be (identical) scalars, and so we can scale $x$ such that $y$ becomes a unitary $u$.\\

\noindent We then find that \[ (1\otimes u) \widetilde{W}_s(1\otimes u^*) = \widetilde{W}_r.\] This implies that there exist two non-equal normal functionals $\omega_1$ and $\omega_2$ on $\widehat{N}$ such that \[(\iota\otimes \omega_1)(\widetilde{W}) = (\iota\otimes \omega_2)(\widetilde{W}).\] As the set $\{(\omega\otimes \iota)(\widetilde{W}) \mid \omega\in B(\mathscr{L}^2(N),\mathscr{L}^2(M))\}$ is $\sigma$-weakly dense in $\widehat{N}$ by Proposition \ref{PropDen2}.2, this clearly gives a contradiction. Hence $\varphi_M(\widetilde{W}_{r,ij}\cdot\widetilde{W}_{s,kl}^*)=0$ for $r\neq s$.

\end{proof}

\begin{Not} By the final part of the previous proposition, we have a unitary transformation \[\mathscr{L}^2(N) \cong \bigoplus_{r\in I_N} \mathscr{H}_r\otimes  \overline{\mathscr{H}_r},\] by means of the map \[\widetilde{W}_{r,ij}^*\xi_{M} \rightarrow T_{r,j}^{1/2}e_{r,i}\otimes \overline{e_{r,j}}.\] In the following, we will then always identify $\mathscr{L}^2(N)$ and $\bigoplus_{r\in I_N} \mathscr{H}_r\otimes  \overline{\mathscr{H}_r}$ in this way, so that for example the elements $x\in N$ act directly as linear operators \[\bigoplus_{r\in I_M} \mathscr{K}_r\otimes  \overline{\mathscr{K}_r}\rightarrow \bigoplus_{r\in I_N} \mathscr{H}_r\otimes  \overline{\mathscr{H}_r}.\]\end{Not}

\begin{Lem}\label{LemFormV} \begin{enumerate}\item With $\widetilde{V}_{r,ij} := T_{r,i}^{1/2} T_{r,j}^{-1/2}\widetilde{W}_{r,ij}^*$, we have the identity \[\widetilde{V} = \sum_{r\in I_N}\sum_{i,j=0}^{n_r-1} \overline{e_{r,ij}} \otimes \widetilde{V}_{r,ij} ,\] the sum converging strongly$^*$.\\
\item The $\widetilde{V}_{r,ij}$ satisfy the following orthogonality relations: \[ \varphi_M(\widetilde{V}_{r,ij}^*\widetilde{V}_{s,kl}) = \delta_{r,s}\delta_{i,k}\delta_{j,l}T_{r,i} \qquad \textrm{for all }r,s\in I_N, 0\leq i,j< n_r, 0\leq k,l<n_s.\]
\item The following equalities hold: \begin{eqnarray*} \widehat{N}' &=& \{(\iota\otimes \omega)(\widetilde{V})\mid \omega\in N_*\}\\ &=& \{x\in B(\mathscr{L}^2(N))\mid \widetilde{W}^*(1\otimes x)\widetilde{W} = 1\otimes x\}.\end{eqnarray*}\end{enumerate}
\end{Lem}

\begin{proof} Choose $r\in I_N$, $0\leq i,j<n_r$ and $\eta\in \mathscr{L}^2(M)$. Then we compute \begin{eqnarray*} \widetilde{V}\, e_{r,i} \otimes \overline{e_{r,j}}\otimes \eta &=& T_{r,j}^{-1/2}\, \widetilde{V} (\widetilde{W}_{r,ij}^*\xi_M\otimes \eta)\\ &=& T_{r,j}^{-1/2}\, \sum_{k=0}^{n_r-1} \widetilde{W}_{r,ik}^*\xi_M\otimes \widetilde{W}_{r,kj}^*\eta \\ &=&  \sum_{k=0}^{n_r-1}  T_{r,k}^{1/2}T_{r,j}^{-1/2}\, e_{r,i}\otimes \overline{e_{r,k}}\otimes \widetilde{W}_{r,kj}^*\eta.\end{eqnarray*} From this, the first point in the lemma follows.\\

\noindent The second point is of course just a reformulation of Theorem \ref{PropOrt}.4.\\

\noindent These orthogonality relations then immediately imply that \[\widehat{N}' = \{(\iota\otimes \omega)(\widetilde{V})\mid \omega\in N_*\}.\] Also the second equality of the third point follows straightforwardly: if $x\in B(\mathscr{L}^2(N))$ and \[\widetilde{W}^*(1\otimes x)\widetilde{W} = 1\otimes x,\] then $x(\omega\otimes \iota)(\widetilde{W}) = (\omega\otimes \iota)(\widetilde{W})x$ for all $\omega\in (N^{\textrm{op}})_*$. From Proposition \ref{PropDen2}.2, we conclude that $x\in \widehat{N}'$. As $\widetilde{W}\in N^{\textrm{op}}\otimes \widehat{N}$, it is also clear that any $x\in \widehat{N}'$ satisfies $\widetilde{W}^*(1\otimes x)\widetilde{W} = 1\otimes x$.\\

\end{proof}

\noindent Recall that we had introduced in Definition \ref{DefCoop} the notion of the co-opposite Galois co-object \[(N^{\textrm{cop}},\Delta_{N^{\textrm{cop}}})=(N,\Delta_N^{\textrm{op}}).\] The following lemma gathers some transfer results between this structure and the original one.

\begin{Lem}\label{LemComp1} \begin{enumerate} \item The right regular $(N,\Delta_{N}^{\textrm{op}})$-corepresentation for $(M,\Delta_M^{\textrm{op}})$ equals $\Sigma \widetilde{W}^*\Sigma$, while the left regular $(N,\Delta_{N}^{\textrm{op}})$-corepresentation equals $\Sigma \widetilde{V}^*\Sigma$.
\item The dual von Neumann algebra $(N^{\textrm{cop}})^{\wedge}$ equals $\widehat{N}'$.
\end{enumerate}

\end{Lem}

\begin{proof} The two statements are easily verified (the second one follows from Lemma \ref{LemFormV}.3). \end{proof}

\noindent One can also relate the two adjoint coactions on respectively $\widehat{N}$ and $\widehat{N}'$, but this result requires some more preparation. We will relegate this investigation to the end of the third section (see Proposition \ref{PropComp2}).\\

\noindent Let us end this section with some remarks on 2-cocycles.

\begin{Def}(\cite{Eno1})\label{DefCoc} Let $(M,\Delta_M)$ be a von Neumann bialgebra. A unitary element $\Omega\in M\bar{\otimes}M$ is called \emph{a unitary 2-cocycle} if $\Omega$ satisfies the following identity, called the 2-cocycle identity: \[(\Omega\otimes 1)(\Delta_M\otimes \iota)(\Omega) = (1\otimes \Omega)(\iota\otimes \Delta_M)(\Omega).\]\end{Def}

\begin{Exa}\label{ExaCocy} Let $(M,\Delta_M)$ be a compact Woronowicz algebra, and $\Omega$ a unitary 2-cocycle for $(M,\Delta_M)$. Then if we put $\mathscr{L}^2(N) = \mathscr{L}^2(M)$, $N=M$ and \[\Delta_N(x) = \Omega\Delta_M(x),\qquad \textrm{for all }x\in M,\] the couple $(N,\Delta_N)$ is a Galois co-object for $(M,\Delta_M)$, called the \emph{Galois co-object associated to $\Omega$}.
\end{Exa}

\noindent It is easy to see that if $\Omega_1$ and $\Omega_2$ are two unitary 2-cocycles for $(M,\Delta_M)$, then their associated Galois co-objects are isomorphic iff the unitary 2-cocycles are coboundary equivalent, that is, iff there exists a unitary $u\in M$ such that \[\Omega_2 = (u^*\otimes u^*)\Omega_1\Delta_M(u).\] In particular, the Galois co-object associated to a 2-cocycle $\Omega$ on $(M,\Delta_M)$ is isomorphic to $(M,\Delta_M)$ as a right Galois co-object iff the 2-cocycle is a \emph{coboundary}, i.e.~ is coboundary equivalent to $1\otimes 1$.

\begin{Def} Let $(M,\Delta_M)$ be a compact von Neumann algebra, and $(N,\Delta_N)$ a Galois co-object for $(M,\Delta_M)$. Then $(N,\Delta_N)$ is called \emph{cleft} if there exists a unitary 2-cocycle $\Omega$ for $(M,\Delta_M)$ such that $(N,\Delta_N)$ is isomorphic to the Galois co-object associated to $\Omega$.\end{Def}

\noindent At the moment, we do not have any examples of non-cleft Galois co-objects for compact Woronowicz algebras, although these \emph{do} exist in the non-compact case. For example, in \cite{Bic1}, non-cleft Galois co-objects were (implicitly) constructed for discrete Woronowicz algebras (see Definition \ref{DefDis}), the Galois co-object being an $l^{\infty}$-direct sum of \emph{rectangular} matrix blocks. For \emph{commutative} compact Woronowicz algebras, that is, those arising from compact groups, it \emph{can} be proven that all Galois co-objects are necessarily cleft (that is, arise from a unitary (measurable) 2-cocycle function on the compact group). We will later prove that this is also the case for \emph{co-commutative} compact Woronowicz algebras (i.e. group von Neumann algebras of discrete groups).\\

\section{Galois objects for discrete Woronowicz algebras}

\noindent In this section, we will make the connection with the theory of Galois objects from \cite{DeC1}.\\

\noindent We first introduce the notion of the dual of a compact Woronowicz algebra.

\begin{Def}\label{DefDis} Let $(M,\Delta_M)$ be a compact Woronowicz algebra with regular left corepresentation $W$. Define \[\widehat{M} = \{(\omega\otimes \iota)(W)\mid \omega\in M_*\}^{\sigma\textrm{-weak closure}}.\] Then $\widehat{M}$ is a von Neumann algebra which can be endowed with a von Neumann bialgebra structure by giving it the unique comultiplication $\Delta_{\widehat{M}}$ such that \[(\iota\otimes \Delta_{\widehat{M}})(W) = W_{13}W_{12}.\] We will call the couple $(\widehat{M},\Delta_{\widehat{M}})$ the \emph{discrete Woronowicz algebra dual to $(M,\Delta_M)$}.\end{Def}

\noindent In fact, we had already introduced the notation $\widehat{M}$ in the remark after Proposition \ref{PropDen2}, as it can be considered to be the space $\widehat{N}$ in the special case where the right Galois co-object $(N,\Delta_N)$ equals $(M,\Delta_M)$. We then also remind that we had introduced some special notations for this case in the Notation \ref{NotWeight}. The following proposition gathers some useful information which can be found in the literature (for example, see the Remark 1.15 in \cite{Vae2}, although we warn the reader that their comultiplication on $\widehat{M}$ is opposite to ours).

\begin{Prop}\label{PropDisc} Let $(M,\Delta_M)$ be a compact Woronowicz algebra, and $(\widehat{M},\Delta_{\widehat{M}})$ its dual. \begin{enumerate}\item For all $r\in I_M$, the number $m_r = \textrm{dim}(\mathscr{K}_r)$ is finite.
\item There exists a \emph{left $\Delta_M$-invariant nsf weight} $\varphi_{\widehat{M}}$ on $\widehat{M}$: for all normal states on $\widehat{M}$ and all positive $x\in \widehat{M}^+$, we have \[\varphi_{\widehat{M}}((\omega\otimes\iota)\Delta_{\widehat{M}}(x)) = \varphi_{\widehat{M}}(x).\] A concrete formula for $\varphi_{\widehat{M}}$ is given by \[\varphi_{\widehat{M}}(e_{r,ij}) = \delta_{i,j}D_{r,j}^{-1}, \qquad \textrm{for all }r\in I_M, 0\leq i,j< m_r.\]
\item On the other hand, define $\psi_{\widehat{M}}$ to be the unique nsf weight on $\widehat{M}$ such that \[\psi_{\widehat{M}}(e_{r,ij}) = \delta_{i,j}\,c_r^2\, D_{r,i},\] where $c_r = \textrm{Tr}(D_r^{-1})^{1/2}$ (which is known as the \emph{quantum dimension} of the irreducible corepresentation corresponding to the index $r\in I_M$). Then $\psi_{\widehat{M}}$ is right $\Delta_{\widehat{M}}$-invariant: for all normal states on $\widehat{M}$ and all positive $x\in \widehat{M}^+$, we have \[\psi_{\widehat{M}}((\iota\otimes\omega)\Delta_{\widehat{M}}(x)) = \psi_{\widehat{M}}(x).\]
\item The Radon-Nikodym derivative between $\psi_{\widehat{M}}$ and $\varphi_{\widehat{M}}$ is given by the (possibly unbounded) positive, non-singular operator \[\delta_{\widehat{M}} = \bigoplus_{r\in I_M} c_r^2D_r^2,\] and $\delta_{\widehat{M}}$ is then a group-like element: for all $t\in \mathbb{R}$, we have \[\Delta_{\widehat{M}}(\delta_{\widehat{M}}^{it}) = \delta_{\widehat{M}}^{it}\otimes \delta_{\widehat{M}}^{it}.\]
\end{enumerate}
\end{Prop}

\noindent The purpose of this section is to show that for an \emph{arbitrary} Galois co-object $(N,\Delta_N)$, the comultiplication $\Delta_{\widehat{M}}$ can be generalized to a coaction $\alpha_{\widehat{N}}$ of $(\widehat{M},\Delta_{\widehat{M}})$ on $\widehat{N}$. This coaction then shares many properties with the actual comultiplication $\Delta_{\widehat{M}}$.\\

\noindent \emph{For the rest of this section, we again fix a compact Woronowicz algebra $(M,\Delta_M)$ and a right Galois co-object $(N,\Delta_N)$ for it. We keep using the notation from the previous section.}\\

\begin{Not}\label{NotGNS} We denote by $\varphi_{\widehat{N}}$ the nsf weight on $\widehat{N}$ which is uniquely determined by the fact that all $e_{r,ij} \in \mathscr{M}_{\varphi_{\widehat{N}}}$, with \[ \varphi_{\widehat{N}}(e_{r,ij}) = \delta_{i,j} T_{r,j}^{-1},\] where the $T_{r,j}$ were introduced in Notation \ref{NotWeight}. We will then take the GNS-construction for $\varphi_{\widehat{N}}$ also inside $\bigoplus_{r\in I_N} \mathscr{H}_r\otimes \overline{\mathscr{H}_r}$, the GNS-map $\Lambda_{\widehat{N}}$ of $\varphi_{\widehat{N}}$ being determined by \[\Lambda_{\widehat{N}}(e_{r,ij}) = T_{r,j}^{-1/2} e_{r,i}\otimes \overline{e_{r,j}}.\]

\noindent The same notation will be used when $(N,\Delta_N)$ equals $(M,\Delta_M)$ considered as a right Galois co-object over itself, taking however into consideration the special notations from Notation \ref{NotWeight}.\end{Not}

\noindent \emph{Remarks:} \begin{enumerate}\item  The fact that there exists a \emph{unique} nsf weight with the above properties requires in fact a small technical argument (at least in case the $\mathscr{H}_r$ are not finite-dimensional). The main observations to make are the well-known fact that any nsf weight $\psi$ on a type $I$-factor is of the form $\textrm{Tr}(S^{1/2}\,\cdot\,S^{1/2})$ for some non-singular positive (possibly unbounded) operator $S$ (see \cite{Tak1}, Lemma VIII.2.8), and the fact that if $\xi$ is a vector with $\psi(\theta_{\xi,\xi}) < \infty$ (where we recall that $\theta_{\xi,\xi}$ is the rank one operator associated to $\xi$), then $\xi \in \mathscr{D}(S^{1/2})$ with $\|S^{1/2}\xi\|^2 = \psi (\theta_{\xi,\xi})$ (this can, for example, be pieced together from the results in \cite{Tak1}, section IX.3). With this information, it should then be easy to verify that the nsf weight $\varphi_{\widehat{N}}$ in the previous notation is indeed well-defined and uniquely determined. \\
\item It is easy to check, using the orthogonality relations between the $\widetilde{W}_{r,ij}$, that for $r\in I_N$ and $0\leq i,j <n_r$, we have \[(\varphi_M(\,\cdot\, \widetilde{W}_{r,ij}^*) \otimes \iota)(\widetilde{W}) \in \mathscr{N}_{\varphi_{\widehat{N}}},\] with \[\Lambda_{\widehat{N}}((\varphi_M(\,\cdot\, \widetilde{W}_{r,ij}^*) \otimes \iota)(\widetilde{W})) = \widetilde{W}_{r,ij}^*\xi_M.\]
Hence our identifications of $\mathscr{L}^2(N)$ and $\mathscr{L}^2(\widehat{N})$ coincide with the `usual' way in which Pontryagin duality is defined in the setting of (locally) compact quantum groups (see \cite{Kus2}).\end{enumerate}

\begin{Prop}\label{PropDualCo} Denote by $\alpha_{\widehat{N}}$ the unital normal faithful $^*$-homomorphism \[\alpha_{\widehat{N}}:\widehat{N}\rightarrow \widehat{N}\bar{\otimes} B(\mathscr{L}^2(M)): x\rightarrow \Sigma \widetilde{W}(x\otimes 1)\widetilde{W}^*\Sigma.\] Then $\alpha_{\widehat{N}}$ has range in $\widehat{N}\bar{\otimes}\widehat{M}$, and determines an \emph{ergodic} coaction of $(\widehat{M},\Delta_{\widehat{M}})$ on $\widehat{N}$.

\end{Prop}

\noindent Before giving the proof, we first state a lemma that we will need in the course of it.

\begin{Lem} \label{LemGen} Take $r\in I_N$ arbitrary. Then the linear span of the set \[\{\widetilde{W}_{r,ij}^{*} a\mid 0\leq i,j< n_r, a\in M\}\] is $\sigma$-weakly dense in $N$.
\end{Lem}

\begin{proof} Denote by $\widetilde{N}$ the $\sigma$-weak closure of the linear span of $\{\widetilde{W}_{r,ij}^{*} a\mid 0\leq i,j< n_r, a\in M\}$. Then clearly $\widetilde{N}\subseteq N$. Now choose $r\in I_N$ fixed, and take an $x\in N$. By Proposition \ref{PropOrt}.2, we have \[x = \sum_{k=0}^{n_r-1} \widetilde{W}_{r,k0}^*(\widetilde{W}_{r,k0}x),\] the sum converging $\sigma$-weakly. As $\widetilde{W}_{r,k0}x\in M$ for all $0\leq k<n_r$, the sum on the right hand side lies in $\widetilde{N}$. So also $N\subseteq \widetilde{N}$.

\end{proof}

\begin{proof}[Proof (of Proposition \ref{PropDualCo})] Take $\omega\in (N^{\textrm{op}})_*$, and denote $x = (\omega\otimes \iota)(\widetilde{W})$. By the pentagonal identity for $\widetilde{W}$ (Proposition \ref{PropElProp}.4), we easily get that \[\alpha_{\widehat{N}}(x) = (\omega\otimes \iota\otimes \iota)(W_{13}\widetilde{W}_{12}) \in \widehat{N}\bar{\otimes}\widehat{M},\] and by an application of the formula $(\iota\otimes \Delta_{\widehat{M}})(W) = W_{13}W_{12}$, we find \[(\alpha_{\widehat{N}} \otimes \iota)\alpha_{\widehat{N}}(x) = (\iota\otimes \Delta_{\widehat{M}})\alpha_{\widehat{N}}(x).\] As elements of the form $x$ constitute a $\sigma$-weakly dense subspace of $\widehat{N}$ by Proposition \ref{PropDen2}.2, we have proven that $\alpha_{\widehat{N}}$ is a well-defined coaction.\\

\noindent We now show that it is ergodic. Take an element $x\in \widehat{N}$ satisfying $\alpha_{\widehat{N}}(x)=x\otimes 1$. Then for all $y\in \widehat{N}$, we get \[(x\otimes 1)\textrm{Ad}_R(y) = \textrm{Ad}_R(y) (x\otimes 1),\] where we recall that $\textrm{Ad}_R(y) = \Sigma \widetilde{W}(1\otimes y)\widetilde{W}^*\Sigma$. As $\{(\iota\otimes \omega)\textrm{Ad}_R(y)\mid \omega\in M_*\}^{''} = \widehat{N}$, a general fact for any coaction of a compact Woronowicz algebra, we find $x\in \mathscr{Z}(\widehat{N})$, the center of $\widehat{N}$.\\

\noindent Write then $x = \sum_{r\in I_N} x_r p_r$, where $r\rightarrow x_r\in l^{\infty}(I_N)$ and $p_r$ the $r$-th minimal central projection of $\widehat{N}$. Then as $\alpha_{\widehat{N}}(x) = x\otimes 1$, we have \[  x (\iota\otimes \omega)(\widetilde{W}^*\Sigma) = (\iota\otimes \omega)(\widetilde{W}^*\Sigma) x \] for all $\omega \in B(\mathscr{L}^2(M),\mathscr{L}^2(N))_*$. If we take $\omega = \omega_{a\xi_M,\widetilde{W}_{r,ij}^*\xi_M}$ for some $r\in I_N$, $0\leq i,j< n_r$ and $a\in M$, and apply both sides of the above equality to $\widetilde{W}_{r,kj}^*\xi_M$ for some $0\leq k<n_r$, we get, by using the orthogonality relations for the $\widetilde{W}_{r,ij}$, that \[  x\widetilde{W}_{r,ki}^*a\xi_M = x_r \widetilde{W}_{r,ki}^*a\xi_M.\] As, with $r$ a \emph{fixed} element of $I_N$, we have that $\{\widetilde{W}_{r,ki}^*a\mid a\in M\}$ is $\sigma$-weakly dense in $N$, by the previous lemma, we get that $x\xi = x_r\xi$ for all $\xi \in \mathscr{L}^2(N)$, so $x$ is a scalar multiple of the unit.

\end{proof}

\noindent Our next goal is to show that $\alpha_{\widehat{N}}$ is nicely behaved with respect to the weight structure on $(\widehat{M},\Delta_{\widehat{M}})$.

\begin{Prop}\label{PropInW} For all $x\in \widehat{N}^+$ and all normal positive states $\omega$ on $\widehat{N}$, we have \[\varphi_{\widehat{M}}((\omega\otimes \iota)\alpha_{\widehat{N}}(x)) = \varphi_{\widehat{N}}(x).\] In particular, $\alpha_{\widehat{N}}$ is an \emph{integrable} coaction.\end{Prop}

\noindent \emph{Remark:} The fact that $\alpha_{\widehat{N}}$ is integrable means that there exists a $\sigma$-weakly dense subspace of $\widehat{N}$ consisting of elements $x$ for which \emph{all} expressions $(\omega\otimes \iota)\alpha_{\widehat{N}}(x)$ with $\omega\in \widehat{N}_*$ lie in $\mathscr{M}_{\varphi_{\widehat{M}}}$.

\begin{proof} We first recall a small technical result from \cite{Vae1}, Proposition 1.3. Namely, as $\alpha_{\widehat{N}}$ is an ergodic coaction, there exists a (not necessarily semi-finite) normal faithful weight $\varphi_{\widehat{N}}'$ on $\widehat{N}$, determined by the following formula: for all $x\in \widehat{N}_+$, we have \[\varphi_{\widehat{N}}'(x) =  \varphi_{\widehat{M}}((\omega\otimes \iota)\alpha_{\widehat{N}}(x)),\] where $\omega$ is \emph{any} normal state on $\widehat{N}$. Our job then is to prove that $\varphi_{\widehat{N}}' = \varphi_{\widehat{N}}$. By the remark after Notation \ref{NotGNS}, it is enough to prove that the $e_{r,ij}$ are in $\mathscr{M}_{\varphi_{\widehat{N}}'}$ with \[\varphi_{\widehat{N}}'(e_{r,ij}) =\delta_{i,j}T_{r,j}^{-1}.\]

\noindent Take $r,s\in I_N$ and $0\leq i,j<n_r$, $0\leq k,l < n_s$. Then we compute, using the GNS-construction for $\varphi_{\widehat{N}}$ from Notation \ref{NotGNS} and the functionals $\omega_{s,kl}$ introduced in Notation \ref{NotWeight}, that \begin{eqnarray*} &&\hspace{-2cm}(\iota\otimes \omega_{s,kl})(\widetilde{W})\Lambda_{\widehat{N}}(e_{r,ij})\\ &=& T_{r,j}^{-1} \widetilde{W}_{s,lk}\widetilde{W}_{r,ij}^*\xi_M \\ &=&  T_{r,j}^{-1} \sum_{t\in I_M} \sum_{m,n=0}^{m_t-1} \langle \widetilde{W}_{s,lk}\widetilde{W}_{r,ij}^*\xi_M,\frac{1}{\|W_{t,mn}^*\xi_M\|} W_{t,mn}^*\xi_M\rangle\, \frac{1}{\|W_{t,mn}^*\xi_M\|} W_{t,mn}^*\xi_M \\ &=& T_{r,j}^{-1} \sum_{t\in I_M} \sum_{m,n=0}^{m_t-1} \varphi_M(W_{t,mn}\widetilde{W}_{s,lk}\widetilde{W}_{r,ij}^*) D_{t,n}^{-1} W_{t,mn}^*\xi_M \\ &=& T_{r,j}^{-1} \sum_{t\in I_M} \sum_{m,n=0}^{m_t-1} \varphi_M(W_{t,mn}\widetilde{W}_{s,lk}\widetilde{W}_{r,ij}^*) \Lambda_{\widehat{M}}(e_{t,mn}),\end{eqnarray*} the latter sums converging in norm.\\

\noindent On the other hand, \begin{eqnarray*} (\omega_{s,kl}\otimes \iota)\alpha_{\widehat{N}}(e_{r,ij}) &=& T_{r,j}^{-1}(\omega_{s,kl} \otimes \iota)\alpha_{\widehat{N}}((\varphi_M(\,\cdot\,\widetilde{W}_{r,ij}^*)\otimes \iota)(\widetilde{W})) \\ &=& T_{r,j}^{-1} (\varphi_M(\,\cdot\,\widetilde{W}_{r,ij}^*)\otimes\omega_{s,kl} \otimes \iota)(W_{13}\widetilde{W}_{12}) \\ &=& T_{r,j}^{-1} \sum_{t\in I_M} \sum_{m,n=0}^{m_t-1} \varphi_M(W_{t,mn}\widetilde{W}_{s,lk}\widetilde{W}_{r,ij}^*)e_{t,mn},\end{eqnarray*} where the latter sum now converges in the strong topology.\\

\noindent As $\Lambda_{\widehat{M}}$ is a strong-norm closed map from $\mathscr{N}_{\varphi_{\widehat{M}}}$ to $\mathscr{L}^2(\widehat{M})$, it follows that \[(\omega_{s,kl}\otimes \iota)(\alpha_{\widehat{N}}(e_{r,ij})) \in \mathscr{N}_{\varphi_{\widehat{M}}},\] with \[\Lambda_{\widehat{M}}((\omega_{s,kl}\otimes \iota)(\alpha_{\widehat{N}}(e_{r,ij}))) = (\iota\otimes \omega_{s,kl})(\widetilde{W})\Lambda_{\widehat{N}}(e_{r,ij}).\] Again by closedness, these assertions remain true when $\omega_{s,kl}$ is replaced by an arbitrary normal functional on $\widehat{N}$.\\

\noindent Let now $\xi_{r,i,j} = e_{r,i}\otimes \overline{e_{r,j}}$ for $r\in I_N$ and $0\leq i,j< n_r$. For any finite subset $\mathcal{J}_0$ of the set $\mathcal{J}=\{(r,i,j)\mid r\in I_N,0\leq i,j< n_r\}$, denote by $P_{\mathcal{J}_0}$ the orthogonal projection onto the linear span of the $\xi_{n}$ with $n\in \mathcal{J}_0$. Take an arbitrary state $\omega \in \widehat{N}_*$ and $x = y^*y$ in the linear span of the $e_{r,ij}$. We remark then that there exists a unit vector $\xi \in \mathscr{L}^2(\widehat{N})$ with $\omega = \omega_{\xi,\xi} (= \langle \,\cdot\,\xi,\xi\rangle)$. We can now compute, using the normality of our weights, that \begin{eqnarray*} \varphi_{\widehat{M}}((\omega\otimes \iota)\alpha_{\widehat{N}}(x)) &=& \lim_{\mathcal{J}_0\underset{\textrm{fin}}{\subseteq} \mathcal{J}} \varphi_{\widehat{M}}((\omega_{\xi,\xi}\otimes \iota)(\alpha_{\widehat{N}}(y)^*(P_{\mathcal{J}_0}\otimes 1)\alpha_{\widehat{N}}(y)))\\ &=& \lim_{\mathcal{J}_0\underset{\textrm{fin}}{\subseteq} \mathcal{J}} \sum_{n\in \mathcal{J}_0} \varphi_{\widehat{M}}((\omega_{\xi,\xi_n}\otimes \iota)(\alpha_{\widehat{N}}(y))^*\cdot (\omega_{\xi,\xi_n}\otimes \iota)(\alpha_{\widehat{N}})(y)) \\ &=& \lim_{\mathcal{J}_0\underset{\textrm{fin}}{\subseteq} \mathcal{J}} \sum_{n\in \mathcal{J}_0} \|(\iota\otimes \omega_{\xi,\xi_n})(\widetilde{W})\Lambda_{\widehat{N}}(y)\|^2 \\ &=&\varphi_{\widehat{N}}(y^*y) = \varphi_{\widehat{N}}(x),\end{eqnarray*} by the unitarity of $\widetilde{W}$. From this, it immediately follows that all $e_{r,ij}$ are integrable for $\varphi_{\widehat{N}}'$, and that $\varphi_{\widehat{N}}' = \varphi_{\widehat{N}}$ on the linear span of the $e_{r,ij}$. This then concludes the proof.

\end{proof}

\noindent We have shown so far that $\alpha_{\widehat{N}}$ is an integrable, ergodic coaction. The final property of $\alpha_{\widehat{N}}$ is that a certain isometry which can be constructed from $\alpha_{\widehat{N}}$ is in fact a unitary.

\begin{Prop} Take $x,y\in \mathscr{N}_{\varphi_{\widehat{N}}}$. Then $\alpha_{\widehat{N}}(y)(x\otimes 1) \in \mathscr{N}_{\varphi_{\widehat{N}}\otimes \varphi_{\widehat{M}}}$, and \[(\Lambda_{\widehat{N}}\otimes \Lambda_{\widehat{M}})(\alpha(y)(x\otimes 1)) = \Sigma \widetilde{W}\Sigma\, \Lambda_{\widehat{N}}(x)\otimes \Lambda_{\widehat{N}}(y).\]

\end{Prop}

\begin{proof} The claim concerning the square integrability of $\alpha_{\widehat{N}}(y)(x\otimes 1)$ follows immediately from the fact that $\alpha_{\widehat{N}}$ is integrable, with $(\iota\otimes \varphi_{\widehat{M}})\alpha_{\widehat{N}}= \varphi_{\widehat{N}}$. Moreover, we can then define an isometry \[\widetilde{G}: \mathscr{L}^2(N) \otimes \mathscr{L}^2(N)\rightarrow \mathscr{L}^2(M)\otimes \mathscr{L}^2(N)\] inside $B(\mathscr{L}^2(N),\mathscr{L}^2(M))\bar{\otimes} \widehat{N}$ such that precisely \[(\Lambda_{\widehat{N}}\otimes \Lambda_{\widehat{M}})(\alpha(y)(x\otimes 1)) = \Sigma \widetilde{G}\Sigma\, \Lambda_{\widehat{N}}(x)\otimes \Lambda_{\widehat{N}}(y)\] for all $x,y\in \mathscr{N}_{\varphi_{\widehat{N}}}$. We need to show that $\widetilde{G} =  \widetilde{W}$.\\

\noindent However, it is easily seen that for all $\omega \in B(\mathscr{L}^2(\widehat{N}))_*$ and $x\in \mathscr{N}_{\varphi_{\widehat{N}}}$, we will have \[ (\iota\otimes \omega)(\widetilde{G})\Lambda_{\widehat{N}}(x) = \Lambda_{\widehat{M}}((\omega\otimes \iota)\alpha_{\widehat{N}}(x)).\] By the computations made in the previous proposition, it follows that $(\iota\otimes \omega)(\widetilde{G})$ coincides with $(\iota\otimes \omega)(\widetilde{W})$ on the linear span of the $e_{r,i}\otimes \overline{e_{r,j}}$ for $r\in I_N$, $0\leq i,j<n_r$, and hence $\widetilde{G}=  \widetilde{W}$.
\end{proof}

\noindent The three propositions above immediately show the following.

\begin{Theorem}\label{TheoGalobj} Let $(M,\Delta_M)$ be a compact Woronowicz algebra, $(N,\Delta_N)$ a right Galois co-object for $(M,\Delta_M)$. Then the couple $(\widehat{N},\alpha_{\widehat{N}})$ makes $\widehat{N}$ into a right Galois object for the discrete Woronowicz algebra $(\widehat{M},\Delta_{\widehat{M}})$, with corresponding Galois unitary $\widetilde{G}=\widetilde{W}$.\end{Theorem}

\noindent For the terminology `Galois object', we refer the reader to \cite{DeC3} (where the notations $N$ and $\widehat{N}$ are interchanged). In fact, it is simply \emph{defined} to be an integrable ergodic coaction for which the map $\widetilde{G}$, as we constructed it in the course of the proof the previous proposition, is a unitary. This map $\widetilde{G}$ is in general called the \emph{Galois unitary} associated with the Galois object, and as we saw in the previous proposition, it coincides precisely with $\widetilde{W}$ in case the Galois object is constructed from a Galois co-object for a compact Woronowicz algebra. Galois objects can also be defined as being \emph{ergodic, semi-dual coactions} (see \cite{Vae1}, Proposition 5.12 for the terminology, and the remark under Proposition 3.5 of \cite{DeC3} for the connection). We further remark that Galois objects for \emph{compact} Woronowicz algebras were treated in \cite{Bic1} (where they are termed `actions of full quantum multiplicity'), and for \emph{commutative} compact Woronowicz algebras, that is for ordinary compact groups, in \cite{Was1} and \cite{Lan1} (where they are termed `actions of full multiplicity'). For Galois objects in the Hopf algebra setting, we refer to the overview \cite{Sch1}.\\

\noindent We may now use the results from \cite{DeC3}, which we gather in the following theorem.

\begin{Theorem}\label{TheoGalob} Let $(N,\Delta_N)$ be a right Galois co-object for a compact Woronowicz algebra $(M,\Delta_M)$, and let $(\widehat{N},\alpha_{\widehat{N}})$ be the associated right Galois object for the dual discrete Woronowicz algebra $(\widehat{M},\Delta_{\widehat{M}})$. Then the following statements hold.
 \begin{enumerate}\item There exists an nsf weight $\psi_{\widehat{N}}$ on $\widehat{N}$, unique up to scaling with a positive constant, which is $\alpha_{\widehat{N}}$-invariant: for all states $\omega\in\widehat{M}_*$ and all $x\in \widehat{N}^+$, we have \[\psi_{\widehat{N}}((\iota\otimes \omega)\alpha_{\widehat{N}}(x)) = \psi_{\widehat{N}}(x).\]
\item The Radon-Nikodym derivative $\delta_{\widehat{N}}\eta \widehat{N}$ of $\psi_{\widehat{N}}$ with respect to $\varphi_{\widehat{N}}$ satisfies \[\alpha_{\widehat{N}}(\delta_{\widehat{N}}^{it}) = \delta_{\widehat{N}}^{it}\otimes \delta_{\widehat{M}}^{it} \qquad \textrm{for all }t\in \mathbb{R}.\]
\item The Radon-Nikodym derivative $\delta_{\widehat{N}}$ is $\sigma^{\varphi_{\widehat{N}}}_t$-invariant, where the latter denotes the modular one-parametergroup associated to $\varphi_{\widehat{N}}$: \[\sigma^{\varphi_{\widehat{N}}}_t(\delta_{\widehat{N}}^{is}) = \delta_{\widehat{N}}^{is} \qquad \textrm{for all }s,t\in \mathbb{R}.\]
\end{enumerate}
\end{Theorem}

\begin{proof} See \cite{DeC3}, Theorem 3.18, Proposition 3.15 and Lemma 3.17.\end{proof}

\begin{Not}\label{NotA} We denote by $A_r$ the non-singular (possibly unbounded) positive operator $p_r\delta_{\widehat{N}} \in B(\mathscr{H}_r)$, so that \[\delta_{\widehat{N}} = \bigoplus_{r\in I_N} A_r.\]\end{Not}

\noindent By the third item in the previous theorem, the operator $A_r$ strongly commutes with $T_r$. In particular, this means that $A_r$ is diagonalizable, and that we may choose our $e_{r,i}\in \mathscr{H}_r$ so that they are \emph{also} eigenvectors for the $A_r$. We then write $A_{r,i}$ for the eigenvalue of $A_r$ with respect to the eigenvector $e_{r,i}$.\\

\noindent \emph{Remark:} Let $\mathscr{A}$ be the Hopf $^*$-algebra associated to $(M,\Delta_M)$, consisting of all elements $x\in M$ with $\Delta_M(x)\in M\odot M$, the algebraic tensor product. Then it is well-known that $\mathscr{A}$ is a $\sigma$-weakly dense sub-$^*$-algebra of $M$, closed under the modular automorphism group $\sigma_t^{\varphi_M}$ of $\varphi_M$. Let $\mathscr{B}_r$ be the sub-$^*$-algebra of $B(\mathscr{H}_r)$ consisting of all elements $x\in B(\mathscr{H}_r)$ with $\textrm{Ad}_R^{(r)}(x) \in B(\mathscr{H}_r)\odot \mathscr{A}$. Again, it is well-known that $\mathscr{B}_r$ is a $\sigma$-weakly dense sub-$^*$-algebra of $B(\mathscr{H}_r)$ (it is the linear span of the coefficients of the spectral subspaces associated to $\textrm{Ad}_R^{(r)}$, see for example \cite{Bic1}). Then the operators $A_r$, introduced in the above notation, turn out to be determined, up to a scalar, by the formula \[ A_r^{it}xA_r^{-it} = (\iota\otimes \varepsilon\circ \sigma_t^{\varphi_M})\textrm{Ad}_R^{(r)}(x),\qquad \textrm{for all }x\in \mathscr{B}_r,\] where $\varepsilon$ denotes the counit of $\mathscr{A}$. This formula can be derived from the way in which $\delta_{\widehat{N}}$ was constructed in \cite{DeC3}. Hence, up to multiplication with a non-singular (possibly unbounded) positive element in the center of $\widehat{N}$, the operator $\delta_{\widehat{N}}$ can be recovered from the knowledge of all the $\textrm{Ad}_R^{(r)}$.\\

\section{Projective representations of compact quantum groups}

\noindent Using the results from the first section, we can easily develop a Peter-Weyl theory for projective representations of compact quantum groups. We will in the following use again the notation which we introduced in the first section. \\

\noindent We first define the notion of a projective representation relative to a fixed Galois co-object.

\begin{Def} Let $(M,\Delta_M)$ be a compact Woronowicz algebra, $(N,\Delta_N)$ a right Galois co-object for $(M,\Delta_M)$. A \emph{(left) $(N,\Delta_N)$-corepresentation} of $(M,\Delta_M)$ on a Hilbert space $\mathscr{H}$ consists of a unitary map $\mathcal{G}\in N\bar{\otimes} B(\mathscr{H})$ such that \[(\Delta_N\otimes \iota)\mathcal{G} = \mathcal{G}_{13}\mathcal{G}_{23}.\]

\noindent We call a Hilbert subspace $\mathscr{K}\subseteq \mathscr{H}$ \emph{invariant} w.r.t.~ the $(N,\Delta_N)$-corepresentation when $\mathcal{G}$ restricts to a unitary in $N\bar{\otimes} B(\mathscr{K})$.\\

\noindent We call $\mathcal{G}$ \emph{irreducible} if the only invariant Hilbert subspaces are $0$ and $\mathscr{H}$, and \emph{indecomposable} when $\mathscr{H}$ can not be written as the direct sum of two non-zero invariant subspaces.\\

\noindent We call two $(N,\Delta_N)$-corepresentations $\mathcal{G}_1$ and $\mathcal{G}_2$ on respective Hilbert spaces $\mathscr{H}_1$ and $\mathscr{H}_2$ \emph{unitarily equivalent} if there exists $u\in B(\mathscr{H}_1,\mathscr{H}_2)$ such that \[\mathcal{G}_2(1\otimes u) = (1\otimes u)\mathcal{G}_1.\]
\end{Def}

\noindent \emph{Remark:} When $(N,\Delta_N)$ comes from a 2-cocycle $\Omega$ for $(M,\Delta_M)$, we will also simply speak of $\Omega$-corepresentations.\\

\begin{Theorem}\label{TheoPW} Let $(M,\Delta_M)$ be a compact quantum group, $(N,\Delta_N)$ a right Galois co-object for $(M,\Delta_M)$. Denote by $\widetilde{V}$ the right regular $(N,\Delta_N)$-corepresentation for $(M,\Delta_M)$, and let \[\widetilde{V}_r = (p_r\otimes 1)\widetilde{V} \in N\bar{\otimes} B(\overline{\mathscr{H}_r})\] be the components of $\widetilde{V}$, where the $p_r$ denote the minimal projections of $\mathscr{Z}(\widehat{N})$.
\begin{enumerate}
\item The unitaries $\Sigma \widetilde{V}_r\Sigma$ are indecomposable left $(N,\Delta_N)$-corepresentations on the Hilbert spaces $\overline{\mathscr{H}_r}$.
\item Any indecomposable $(N,\Delta_N)$-corepresentation is unitarily equivalent with a unique $\Sigma\widetilde{V}_r\Sigma$.
\item Any $(N,\Delta_N)$-corepresentation splits as a direct sum of indecomposable $(N,\Delta_N)$-corepresentations.
\item Any indecomposable $(N,\Delta_N)$-corepresentation is irreducible.
\end{enumerate}
\end{Theorem}

\begin{proof} As $(\iota\otimes \Delta_N)\widetilde{V} = \widetilde{V}_{12}\widetilde{V}_{13}$, we immediately get that the unitaries $\Sigma\widetilde{V}_r\Sigma$ are left $(N,\Delta_N)$-corepresentations. By Lemma \ref{LemFormV}.3, the space \[\{(\iota\otimes \omega)(\widetilde{V}_r)\mid \omega\in B(\mathscr{L}^2(M),\mathscr{L}^2(N))_*\}\] equals the whole of $B(\overline{\mathscr{H}_r})$, from which it immediately follows that $\Sigma \widetilde{V}_r\Sigma$ is indecomposable, and even irreducible.\\

\noindent For the second statement, we use that the linear span of the matrix entries of the $\widetilde{V}_r$'s span a norm-dense subset of $\mathscr{L}^2(N)$ when applied to $\xi_M$. Hence, if $\mathcal{G}$ is an indecomposable left $(N,\Delta_N)$-corepresentation of $(M,\Delta_M)$ on a Hilbert space $\mathscr{H}$, there must exist some $r\in I_N$ and an $x\in B(\mathscr{H},\overline{\mathscr{H}_r})$ such that \[(\varphi_M\otimes \iota)( (\Sigma\widetilde{V}_r\Sigma)^* (1\otimes x)\mathcal{G}) \neq 0.\] As in the proof of Theorem \ref{PropOrt}, this forces a scalar multiple of this expression to be a unitary intertwiner, proving that $\mathcal{G}$ is isomorphic to $\Sigma\widetilde{V}_r\Sigma$. By the orthogonality relations between the $\Sigma\widetilde{V}_r\Sigma$, these are all pairwise non-isomorphic. Hence the above $r$ for $\mathcal{G}$ is uniquely determined.\\

\noindent To prove the third statement, consider the normal faithful unital $^*$-homomorphism \[\alpha: B(\mathscr{H})\rightarrow M\bar{\otimes} B(\mathscr{H}): x\rightarrow \mathcal{G}^*(1\otimes x)\mathcal{G}.\] As $(\Delta_N\otimes\iota)\mathcal{G}=\mathcal{G}_{13}\mathcal{G}_{23}$, we see that $\alpha$ is a coaction. Let $Z = B(\mathscr{H})^{\alpha}$, the set of fixed points for $\alpha$. As in the proof of Proposition \ref{PropDen2}, we have that $Z$ is the range of a normal conditional expectation on $B(\mathscr{H})$. Hence $Z$ is a von Neumann algebraic direct sum of type $I$-factors. Let then $A$ be an atomic maximal abelian von Neumann subalgebra of $Z$, and denote by $p_i$ the set of minimal projections in $A$. Then it is clear that each $p_i\mathscr{H}$ is a fixed subspace for $\mathcal{G}$. The spaces $p_i\mathscr{H}$ must further be indecomposable: for if not, then we could find a $p_i$ and a non-zero projection $p$ in $B(\mathscr{H})$ with $p$ strictly smaller than $p_i$ and both $p\mathscr{H}$ and $(p_i-p)\mathscr{H}$ invariant under $\mathcal{G}$. This would imply that $p$ is a fixed element for $\alpha$, commuting with all $x\in A$. Hence $p\in A$ by maximal abelianness. As $p_i$ was a minimal projection in $A$, this gives a contradiction.\\

\noindent As for the fourth point, we may take our indecomposable $(N,\Delta_N)$-corepresentation to be some $\Sigma\widetilde{V}_r\Sigma$, for which we have already proven irreducibility in the proof of the first point.

\end{proof}

\begin{Cor}\label{CorFormN} Let $(N,\Delta_N)$ be a right Galois co-object for a compact Woronowicz algebra $(M,\Delta_M)$. Let $\mathcal{G}_{i}$, $i\in I$, be a maximal set of non-isomorphic irreducible $(N,\Delta_N)$-corepresentations on Hilbert spaces $\mathscr{H}_i$. Then $I \cong I_N$, and $\widehat{N}\cong \oplus_{i\in I} B(\mathscr{H}_i)$.
\end{Cor}
\begin{proof} This follows immediately from the second point of the previous proposition.

\end{proof}

\noindent We can now pass to projective representations without reference to a fixed Galois co-object.

\begin{Def} Let $(M,\Delta_M)$ be a compact Woronowicz algebra. A \emph{projective (left) corepresentation of $(M,\Delta_M)$ on a Hilbert space $\mathscr{H}$} consists of a left coaction $\alpha$ of $(M,\Delta_M)$ on $B(\mathscr{H})$, \[\alpha: B(\mathscr{H})\rightarrow M\bar{\otimes} B(\mathscr{H}).\]
\end{Def}

\noindent \emph{Remarks:} \begin{enumerate}\item Interpreting $(M,\Delta_M)$ as the space of $\mathscr{L}^{\infty}$-functions on some `compact quantum group' $\mathbb{G}$, the above then corresponds to having a (necessarily continuous) action of $\mathbb{G}$ on $B(\mathscr{H})$. As $\textrm{Aut}(B(\mathscr{H}))\cong \mathcal{U}(\mathscr{H})/S^1$, this indeed captures the notion of a projective representation when $\mathbb{G}$ is an actual compact group.
\item One similarly has the notion of a \emph{projective right corepresentation} of $(M,\Delta_M)$, for which we replace the left coaction $\alpha$ above by a right coaction. For example, the coactions $\textrm{Ad}_R$ and $\textrm{Ad}_R^{(r)}$ from the first section are then projective right corepresentations. At the end of this section, we will show one can pass from left to right projective representations, so that one may essentially restrict oneself to the study of projective left corepresentations, as we will do.
\item Some results on (special) coactions of compact Kac algebras on type $I$ factors appear in \cite{Lin1}.
\end{enumerate}

\noindent In \cite{DeC1}, we proved that from any projective corepresentation $\alpha$, one can construct a Galois co-object $(N,\Delta_N)$ together with an $(N,\Delta_N)$-corepresentation $\mathcal{G}$ `implementing' $\alpha$. We will state the proposition and give a sketch of the proof. For the full proof, we refer the reader to \cite{DeC1}.

\begin{Prop}\label{PropRecon} Let $(M,\Delta_M)$ be a compact Woronowicz algebra, $\mathscr{H}$ a Hilbert space, and $\alpha$ a projective corepresentation of $(M,\Delta_M)$ on $\mathscr{H}$. Then there exists a right Galois co-object $(N,\Delta_N)$ for $(M,\Delta_M)$, together with an $(N,\Delta_N)$-corepresentation $\mathcal{G}$ of $(M,\Delta_M)$ on $\mathscr{H}$, such that \[\mathcal{G}^*(1\otimes x)\mathcal{G} = \alpha(x),\qquad \textrm{for all }x\in B(\mathscr{H}).\] \end{Prop}

\begin{proof}[Sketch of proof] Choose a basis $\{e_i\}_{i\in I}$ of $\mathscr{H}$, and fix an element $0\in I$. We can then consider the Hilbert space $\mathscr{L}^2(N) = \alpha(e_{00})(\mathscr{L}^2(M)\otimes \mathscr{H})$. We can construct a unitary \[\mathcal{G}: \mathscr{L}^2(M)\otimes \mathscr{H}\rightarrow \mathscr{L}^2(N)\otimes \mathscr{H}: \xi \rightarrow \sum_{i\in I} (\alpha(e_{0i})\xi)\otimes e_i.\] Denote by $\mathcal{G}_{ij}$ the $i,j$-th component of $\mathcal{G}$. Then $\mathcal{G}_{ij}$ is an operator from $\mathscr{L}^2(M)$ to $\mathscr{L}^2(N)$. We define \[N = \{\mathcal{G}_{ij} m\mid i,j\in I, m\in M\}^{\sigma\textrm{-weakly closed linear span}}.\]

\noindent It is then possible to construct a map $\Delta_N:N\rightarrow N\bar{\otimes} N$, uniquely determined by the properties that \[(\Delta_N\otimes \iota)\mathcal{G} = \mathcal{G}_{[13]}\mathcal{G}_{[23]}\] (where we have added brackets in the leg numbering notation to distinguish them from the indices for matrix coefficients of $\mathcal{G}$) and \[\Delta_N(xy) = \Delta_N(x)\Delta_{M}(y), \qquad \textrm{for all }x\in N,y\in M.\]

\noindent One proves that $(N,\Delta_N)$ is a Galois co-object for $(M,\Delta_M)$, and then it immediately follows from the above property that $\mathcal{G}$ is a left $(N,\Delta_N)$-corepresentation of $(M,\Delta_M)$ on $\mathscr{H}$. Finally, one proves that $\mathcal{G}^*(1\otimes x)\mathcal{G} = \alpha(x)$ by direct computation.
\end{proof}

\begin{Def}\label{DefImpl} Let $\alpha$ be a projective corepresentation of a compact Woronowicz algebra $(M,\Delta_M)$ on a Hilbert space $\mathscr{H}$. Denote by $(N,\Delta_N)$ the Galois co-object constructed from $\alpha$ as in the above proposition, and denote by $\lbrack (N,\Delta_N)\rbrack$ its isomorphism class. Then we say that $\alpha$ is an $\lbrack (N,\Delta_N)\rbrack$-corepresentation.\end{Def}

\noindent It can be proven (see Proposition 3.4 in \cite{DeC1}) that if $\alpha$ is an $\lbrack (N,\Delta_N)\rbrack$-corepresentation of $(M,\Delta_M)$ on a Hilbert space $\mathscr{H}$, and if there exists a Galois co-object $(\widetilde{N},\Delta_{\widetilde{N}})$ for $(M,\Delta_M)$ which possesses an $(\widetilde{N},\Delta_{\widetilde{N}})$-corepresentation on $\mathscr{H}$ implementing $\alpha$, then necessarily $\lbrack (\widetilde{N},\Delta_{\widetilde{N}})\rbrack = \lbrack (N,\Delta_N)\rbrack$. One may regard the isomorphism class of such a Galois co-object as a generalization of the notion of a 2-cohomology class. Also remark that if $\mathcal{G}\in N\bar{\otimes} B(\mathscr{H})$ is a projective $(N,\Delta_N)$-corepresentation of $(M,\Delta_M)$ on a Hilbert space $\mathscr{H}$, then the associated projective corepresentation \[\alpha: B(\mathscr{H})\rightarrow M\bar{\otimes}B(\mathscr{H}): x\rightarrow \mathcal{G}^*(1\otimes x)\mathcal{G}\] is an $\lbrack{(N,\Delta_N)}\rbrack$-corepresentation by the above uniqueness result.\\

\noindent If $(N,\Delta_N)$ is a Galois co-object for a compact Woronowicz algebra $(M,\Delta_M)$, and an $\lbrack (N,\Delta_N)\rbrack$-corepresentation $\alpha$ for $(M,\Delta_M)$ is given, then it is in general not true that all $(N,\Delta_N)$-corepresenta-tions implementing $\alpha$ are isomorphic: consider for example ordinary one-dimensional representations. We will see a further instance of this in the final section.\\

\begin{Def} Let $(M,\Delta_M)$ be a compact Woronowicz algebra, $\alpha$ a projective corepresentation of $(M,\Delta_M)$ on a Hilbert space $\mathscr{H}$. We say that $\alpha$ is an \emph{irreducible} projective corepresentation if $\alpha$ is ergodic.
\end{Def}

\begin{Prop}\label{PropIndec} Let $(N,\Delta_N)$ be a right Galois co-object for a compact Woronowicz algebra $(M,\Delta_M)$, let $\alpha$ be an $\lbrack (N,\Delta_N)\rbrack$-corepresentation of $(M,\Delta_M)$ on a Hilbert space $\mathscr{H}$, and let $\mathcal{G}$ be an $(N,\Delta_N)$-corepresentation implementing $\alpha$.\\

\noindent Then there is a one-to-one correspondence between the set of $\alpha$-fixed self-adjoint projections in $B(\mathscr{H})$ and the $\mathcal{G}$-invariant subspaces $\mathscr{K}$ of $\mathscr{H}$, given by the correspondence \[p\rightarrow \mathscr{K} = p\mathscr{H}.\]

\noindent In particular, $\alpha$ is irreducible iff $\mathcal{G}$ is irreducible.\end{Prop}

\begin{proof} By assumption, we have that \[\alpha(x) = \mathcal{G}^*(1\otimes x)\mathcal{G} \qquad \textrm{for all }x\in B(\mathscr{H}).\] So if $p$ is a self-adjoint projection in $B(\mathscr{H})$ with $\alpha(p)= 1\otimes p$, we have \[\mathcal{G}(1\otimes p )  = (1\otimes p)\mathcal{G},\] and hence \[\mathcal{G} (\mathscr{L}^2(M)\otimes p\mathscr{H}) \subseteq \mathscr{L}^2(N)\otimes p\mathscr{H},\] which means $p\mathscr{H}$ is a $\mathcal{G}$-invariant subspace.\\

\noindent Conversely, if $\mathscr{K}$ is a $\mathcal{G}$-invariant subspace, and $p$ the projection onto $\mathscr{K}$, then also $\mathscr{K}^{\perp}$ is $\mathcal{G}$-invariant by Theorem \ref{TheoPW}. Hence we have \[\mathcal{G}(1\otimes p) = (1\otimes p)\mathcal{G},\] and so $\alpha(p)=1\otimes p$, i.e.~ $p$ is an $\alpha$-fixed projection.\end{proof}

\noindent \emph{Remark:} We in fact already used the above argument in the course of proving Theorem \ref{TheoPW}.3.\\

\noindent Our next proposition shows how projective corepresentations and ordinary corepresentations mesh together.\\

\begin{Prop}\label{PropFus} Let $(N,\Delta_N)$ be a right Galois co-object for a compact Woronowicz algebra $(M,\Delta_M)$. Let $\mathcal{G}$ be a projective left $(N,\Delta_N)$-corepresentation on a Hilbert space $\mathscr{H}$, and let $U$ be an ordinary left corepresentation of $(M,\Delta_M)$ on a Hilbert space $\mathscr{K}$. \\

\noindent Then the following statements hold.
\begin{enumerate} \item The unitary $\mathcal{G}_{12}U_{13}\in N\bar{\otimes}B(\mathscr{H}\otimes \mathscr{K})$ is a unitary $(N,\Delta_N)$-corepresentation on $\mathscr{H}\otimes \mathscr{K}$.
\item If both $\mathcal{G}$ and $U$ are irreducible, then $\mathcal{G}_{12}\mathcal{U}_{13}$ is a \emph{finite} direct sum of irreducible $(N,\Delta_N)$-corepresentations.
\end{enumerate}
\end{Prop}

\begin{proof} The fact that $\mathcal{G}_{12}U_{13}$ is a unitary $(N,\Delta_N)$-corepresentation is trivial to verify. If further $\mathcal{G}$ and $U$ are irreducible, then we know already that $\mathcal{G}_{12}U_{13} = \oplus_{i\in J} \mathcal{G}_j$ for a certain set $\mathcal{G}_j$ of irreducible $(N,\Delta_N)$-corepresentations indexed by a parameter set $J$. We have to prove that $J$ is finite.\\

\noindent Let $\alpha$ be the projective corepresentation associated to $\mathcal{G}$, so \[\alpha(x) = \mathcal{G}^*(1\otimes x)\mathcal{G}, \qquad x\in B(\mathscr{H}).\] By the previous proposition, we know that $\alpha$ is ergodic. Let $\mathscr{B}$ be the linear span of the spectral subspaces inside $B(\mathscr{H})$, which is a $\sigma$-weakly dense sub-$^*$-algebra of $B(\mathscr{H})$ (see the remark following Notation \ref{NotA}). If we then denote by $U_r$, $r\in I_{M}$, a total set of representatives for the irreducible corepresentations of $(M,\Delta_M)$ on Hilbert spaces $\mathscr{K}_r$, we know by \cite{Boc1} that \[(\mathscr{B},\alpha) \cong (\oplus_{r\in I_M} \mathscr{K}_r\otimes \mathbb{C}^{k_r},\oplus_{r\in I_{M}} U_r\otimes 1)\] as a comodule over the Hopf algebra $\mathscr{A}\subseteq M$, \emph{where $k_r <\infty$}.\\

\noindent Now if $\beta$ is the projective corepresentation associated to $\mathcal{G}_{12}U_{13}$, then \[\beta(x) = U_{13}^*(\alpha\otimes \iota)(x)U_{13}, \qquad \textrm{for all }x\in B(\mathscr{H})\otimes B(\mathscr{K}).\] Hence if $\widetilde{\mathscr{B}}$ is the linear span of the spectral subspaces of $\beta$, then as a comodule, we have \[\widetilde{\mathscr{B}} \cong U^{c}\times \mathscr{B}\times U,\] where $U^c$ denotes the contragredient of $U$ and where we denote by $\times$ the tensor product of corepresentations/comodules. But this means that the trivial corepresentation appears in $\widetilde{B}$ with multiplicity $\sum g_rk_r$, where $g_r$ is the multiplicity of $U_r \subseteq U \times U^c$. Hence the fixed point algebra of $\beta$ is finite-dimensional, and by the previous Proposition, $J$ will have as its cardinality the dimension of a maximal abelian subalgebra of the fixed point algebra of $\beta$. Hence $J$ is finite.
\end{proof}

\noindent The previous proposition leads to the following considerations. Let $(N,\Delta_N)$ be a fixed right Galois co-object for a compact Woronowicz algebra $(M,\Delta_M)$. Then we can make a W$^*$-category $\mathscr{D}$ by considering as objects the $(N,\Delta_N)$-corepresentations which are (isomorphic to) finite direct sums of irreducible $(N,\Delta_N)$-corepresentations, and as morphisms bounded intertwiners. Then if we denote by $\mathscr{C}$ the tensor W$^*$-category of finite-dimensional $(M,\Delta_M)$-corepresentations, we can make $\mathscr{D}$ into a right $\mathscr{C}$-module by the natural composition introduced above: \[ \mathscr{D}\times \mathscr{C}\rightarrow \mathscr{D}: (\mathcal{G},U)\rightarrow \mathcal{G}_{12}U_{13},\] while the action of morphisms is simply by tensoring. We can then also turn $\textrm{F}_N := \oplus_{r\in I_N} \mathbb{Z}$ into a module over the fusion ring $\textrm{F}_M := \oplus_{r\in I_M} \mathbb{Z}$ by means of the fusion rules associated to this categorical construction.\\

\noindent But in fact, there is another different way in which to obtain these fusion rules, making use of the theory developed in \cite{Tom1}. In that paper, Wassermann's multiplicity theory for ergodic compact Lie group actions on C$^*$-and von Neumann algebras is extended to the setting of compact Woronowicz algebras. Although the paper works in the C$^*$-algebraic realm and uses right coactions, the results also apply in the von Neumann algebra setting and with left actions, and we make the transition without further comment in explaining these ideas.\\

\noindent Let then $(M,\Delta_M)$ be a compact Woronowicz algebra with an ergodic coaction $\alpha$ on a von Neumann algebra $A$. It is well-known that the crossed product \[ M\ltimes A = \{(x\otimes 1)\alpha(y)\mid x\in M, y\in A\}'' \subseteq B(\mathscr{L}^2(M)\otimes \mathscr{L}^2(A))\] is a von Neumann algebraic direct sum of type $I$ factors. Let $I_{\alpha}$ be the set of atoms of the center of $M\ltimes A$, and let $\textrm{F}_{\alpha}$ be the free abelian group generated by $I_{\alpha}$. Then one can turn $\textrm{F}_{\alpha}$ into a right $\textrm{F}_M$-module by the following procedure. Let $\{p_s \mid s\in I_{\alpha}\}$ be the set of minimal projections in the center of $M\ltimes A$, and choose for each $s\in I_{\alpha}$ a minimal projection $e_s\leq p_s$ in $M\ltimes A$. We can equip the corners $e_s (B(\mathscr{L}^2(M))\bar{\otimes} A)e_t$ with a left $(M,\Delta_M)$-coaction $\alpha_{st}$ by the formula \[\alpha_{st}(z) = (\Sigma\otimes 1)(V^*_{12} (\iota\otimes \alpha)(z)V_{12})(\Sigma\otimes 1).\] For each $r\in I_M$, $s,t\in I_{\alpha}$, define $M^{(r)}_{st}$ to be the dimension of the set of $(M,\Delta_M)$-intertwiners between the corepresentation $U_r$ associated to $r$ and $\alpha_{st}$. Then the action of $r\in I_M$ on an element $t\in I_{\alpha}$ is defined as \[ t\cdot r := \sum_{s\in I_{\alpha}} M^{(r)}_{ts}\cdot s.\]

\noindent Let now $(N,\Delta_N)$ be a right Galois co-object for the compact Woronowicz algebra $(M,\Delta_M)$. Choose $r\in I_N$. Then we can apply the above ideas to the left coaction $\alpha_r$ on $B(\mathscr{H}_r)$, where $\alpha_r$ is the coaction associated to the irreducible projective $(N,\Delta_N)$-corepresentation $\mathcal{G}_r$ pertaining to $r$. \emph{We claim that the resulting right $\textrm{F}_M$-module is independent of the choice of $r$, and coincides precisely with the right $\textrm{F}_M$-module as constructed after Proposition \ref{PropFus}}. We will briefly indicate how this can be proven.\\

\noindent We first observe that $M\ltimes B(\mathscr{H}_r)$ equals $\mathcal{G}_r^* (\widehat{N}\bar{\otimes} B(\mathscr{H}_r))\mathcal{G}_r$. Indeed, this follows by the characterization of $\widehat{N}$ as a fixed point space, by the pentagon identity for $\widetilde{V}$ (and the related pentagon identity for $\mathcal{G}_r$), by the fact that $\mathcal{G}_r$ implements $\alpha_r$, and finally by the characterization of $M\ltimes B(\mathscr{H}_r)$ as the set of elements $z$ in $B(\mathscr{L}^2(M))\bar{\otimes} B(\mathscr{H}_r)$ satisfying \[ V_{12} z_{13}V_{12}^* = (\iota\otimes \alpha_r)(z).\] Hence we already see that $I_{\alpha_r} = I_N$.\\

\noindent Choose then for each $s\in I_N$ the element \[e_s := e_{s,00}\otimes e_{r,00} \in \widehat{N}\bar{\otimes} B(\mathscr{H}_r) \cong M\ltimes B(\mathscr{H}_r)\] as a minimal projection. Then by transporting all structure with the aid of $\mathcal{G}$, one sees that the corner $e_s (B(\mathscr{L}^2(M))\bar{\otimes} B(\mathscr{H}_r))e_t$ is isomorphic to $e_{s,00} B(\mathscr{L}^2(N))e_{t,00}$, equipped with the restriction of the coaction $\textrm{Ad}_L$ (which appears in the proof of Proposition \ref{PropDen2}). This may further be simplified to the coaction \[\alpha_{st}: B(\mathscr{H}_s,\mathscr{H}_t)\rightarrow M\bar{\otimes} B(\mathscr{H}_s,\mathscr{H}_t): x\rightarrow \mathcal{G}_s^* (1\otimes x)\mathcal{G}_t.\] This final coaction may be interpreted as corresponding to the (ordinary) corepresentation `$\mathcal{G}_s^{c}\times \mathcal{G}_t$'. A Frobenius-type argument then shows that this corepresentation contains some $U_u$ with $u\in I_M$ as much times as $\mathcal{G}_t$ is contained in $\mathcal{G}_s\times U_u$. This shows that the two mentioned fusion rules indeed coincide, and ends our sketch of proof.\\

\noindent To end this section, let us come back to comparing the structures of $(N,\Delta_N)$ and $(N,\Delta_N^{\textrm{op}})$ which we started in Lemma \ref{LemComp1}. We begin by introducing a certain antipode on a subspace of $N$.

\begin{Prop} Let $(N,\Delta_N)$ be a right Galois co-object for a compact quantum group $(M,\Delta_M)$. Denote by $\mathscr{N}$ the linear span of the $\widetilde{V}_{r,ij}$ in $N$ (see Lemma \ref{LemFormV}). Denote by $\mathscr{A}$ the corresponding subspace of $M$, which coincides with the Hopf $^*$-algebra associated with $(M,\Delta_M)$ (see the remark after Notation \ref{NotA}). Then the following statements hold.

\begin{enumerate} \item The space $\mathscr{N}$ is a right $\mathscr{A}$-module.

\item If we define the anti-linear map \[S_N(\,\cdot\,)^* : \mathscr{N}\rightarrow \mathscr{N}: \widetilde{V}_{r,ij}\rightarrow \widetilde{V}_{r,ji},\] then for all $x\in \mathscr{N}$ and $y\in \mathscr{A}$, we have \[S_N(xy)^* = S_N(x)^*S_M(y)^*,\] where $S_M$ denotes the antipode of the Hopf $^*$-algebra $\mathscr{A}$.

\item For all $r\in I_N$ and $0\leq i,j<n_r$, we have \[\Lambda_N(S_N(\widetilde{V}_{r,ij})^*) = \left(\frac{T_{r,j}}{T_{r,i}}\right)^{1/2} J_{\widehat{N}} \Lambda_N(\widetilde{V}_{r,ij}),\] where $J_{\widehat{N}}$ denotes the modular conjugation for $\widehat{N}$, given by $e_{r,i}\otimes \overline{e_{r,j}}\rightarrow e_{r,j}\otimes \overline{e_{r,i}}$.

\end{enumerate}

\end{Prop}

\begin{proof} As the $\widetilde{V}_{r,ij}$ form a basis of $\mathscr{N}$, it is easy to see that $S_N(\,\cdot\,)^*$ is well-defined. Moreover, using the formulas in Lemma \ref{LemFormV}.1, the third statement follows immediately. As for the first point, this is an immediate consequence of Proposition \ref{PropFus}.2.\\

\noindent So the only thing left to show is the second statement, which is at least meaningful by the first part of the proposition.\\

\noindent For $\omega$ in the predual of $\widehat{N}'$, denote by $\overline{\omega}$ the normal functional $x\rightarrow \overline{\omega(x^*)}$ on $\widehat{N}'$. Let us call a normal functional on $\widehat{N}'$ \emph{elementary} when it is of the form $\overline{e_{r,ij}}\rightarrow \omega(e_{r,ji})$ for some $\omega$ in the linear span of the $\omega_{s,kl}$ (see Notation \ref{NotWeight}). Then with $\omega$ elementary, we immediately obtain the formula \begin{equation}\label{FormS} S_N((\omega\otimes \iota)(\widetilde{V}))^* = (\overline{\omega}\otimes \iota)(\widetilde{V}).\end{equation}

\noindent Choose now normal functionals $\omega_1$ and $\omega_2$ on respectively $B(\mathscr{L}^2(N))$ and $B(\mathscr{L}^2(M))$ which restrict to elementary functionals on respectively $\widehat{N}'$ and $\widehat{M}'$. Then by the pentagonal identity for $\widetilde{V}$, we have \[(\omega_1\otimes \iota)(\widetilde{V})(\omega_2\otimes \iota)(V)  = (\widetilde{\omega}\otimes \iota)(\widetilde{V}),\] where $\widetilde{\omega}$ is the functional \[x\in B(\mathscr{L}^2(N))\rightarrow (\omega_1\otimes \omega_2)(\widetilde{V}_{12}^*(1\otimes x)\widetilde{V}_{12}).\] By the first part of the proposition, the restriction of $\widetilde{\omega}$ to $\widehat{N}'$ is again elementary. Combining these statements with equation (\ref{FormS}) (and the corresponding one for $S_M$), we see that $S_N(\,\cdot\,)^*$ is indeed right $S_M(\,\cdot\,)^*$-linear.

\end{proof}

\noindent We can now make the connection between the adjoint coactions of $(M,\Delta_M)$ on $\widehat{N}$ and $(M,\Delta_M^{\textrm{op}})$ on $\widehat{N}'$ respectively (see the remark after Lemma \ref{LemComp1}). Let us first recall that any compact Woronowicz algebra $(M,\Delta_M)$ is endowed with an involutive anti-comultiplicative anti-$^*$-automorphism $R_M$, given by the formula \[R_M(x) = J_{\widehat{M}}x^*J_{\widehat{M}}\qquad \textrm{for all }x\in M.\] More concretely, we have $R_M(W_{r,ij})=V_{r,ji}$ for all $r\in I_M$ and $0\leq i,j<m_r$. We will also denote \[\mathscr{C}_{\widehat{N}}: \widehat{N}\rightarrow \widehat{N}': x\rightarrow J_{\widehat{N}}x^*J_{\widehat{N}} = \overline{x^*},\] and use the same notation for its inverse.

\begin{Prop}\label{PropComp2} Let $(N,\Delta_N)$ be a right Galois co-object for a compact Woronowicz algebra $(M,\Delta_M)$. Let $(N,\Delta_N^{\textrm{op}})$ be the co-opposite right Galois co-object for $(M,\Delta_M^{\textrm{op}})$. Then the right adjoint coaction of $(M,\Delta_M^{\textrm{op}})$ on $(N^{\textrm{cop}})^{\wedge}= \widehat{N}'$ is given by \[x\rightarrow (\mathscr{C}_N\otimes R_M)Ad_R(\mathscr{C}_N(x)).\]
\end{Prop}

\begin{proof} Note that the right adjoint coaction on $B(\overline{\mathscr{H}_r})$ is, by its definition and by Lemma \ref{LemComp1}.1, given as \[x\rightarrow \widetilde{V}^*(x\otimes 1)\widetilde{V}.\] Denote now by $J_{\widehat{N}}$ the modular conjugation for $\widehat{N}$, which we recall is simply the map \[\mathscr{L}^2(N)\rightarrow \mathscr{L}^2(N): e_{r,i}\otimes \overline{e_{r,j}}\rightarrow e_{r,j}\otimes \overline{e_{r,j}}.\] It is then easily seen that the proposition follows one we can prove the following identity: \begin{equation}\label{EqnIdV}\widetilde{V} = (J_{\widehat{N}}\otimes J_{\widehat{N}})(\Sigma \widetilde{W}^*\Sigma)(J_{\widehat{N}}\otimes J_{\widehat{M}}).\end{equation}

\noindent Now piecewise, the identity (\ref{EqnIdV}) corresponds to the identities \[J_{\widehat{N}}\widetilde{W}_{r,ji}^*J_{\widehat{M}} = \widetilde{V}_{r,ij}, \qquad \textrm{for all }r\in I_N,0\leq i,j<n_r.\] But in \cite{DeC3}, it was proven that $J_{\widehat{N}}xJ_{\widehat{M}} \in N$ for $x\in N$ (see the remark just before Lemma 4.3 in that paper). Hence we only have to check if \[J_{\widehat{N}}\widetilde{W}_{r,ji}^*J_{\widehat{M}}\xi_M = \widetilde{V}_{r,ij}\xi_M, \textrm{for all }r\in I_N,0\leq i,j<n_r.\]
This now follows from an easy computation using Lemma \ref{LemFormV}.
\end{proof}

\noindent \emph{Remark:} It seems nicer to treat the right adjoint $(M,\Delta_M^{\textrm{op}})$-coaction on $\widehat{N}'$ as a left $(M,\Delta_M)$-coaction: \[\textrm{Ad}_L: \widehat{N}'\rightarrow M\bar{\otimes} \widehat{N}': x\rightarrow \Sigma \widetilde{V}^*(x\otimes 1)\widetilde{V}\Sigma.\] These then localize to left adjoint coactions $\textrm{Ad}_L^{(r)}$ on the $B(\overline{\mathscr{H}_r})$. Note that the map $\textrm{Ad}_L$ (as well as the map $S_N$) in fact already appeared in the proof of Proposition \ref{PropDen2}, and that the $\textrm{Ad}_L^{(r)}$ are nothing but the $\lbrack (N,\Delta_N)\rbrack$-corepresentations of $(M,\Delta_M)$ associated to the $(N,\Delta_N)$-corepresentations $\Sigma \widetilde{V}_r\Sigma$ from Theorem \ref{TheoPW}.

\section{Reflecting a compact Woronowicz algebra across a Galois co-object}

\noindent In this section, we will consider in the special case of compact Woronowicz algebras a technique which was introduced in \cite{DeC1}. The following theorem was proven in \cite{DeC1}, Proposition 2.1 and Theorem 0.7.

\begin{Theorem} \label{TheoRefl} Let $(M,\Delta_M)$ be a compact Woronowicz algebra, and $(N,\Delta_N)$ a right Galois co-object for $(M,\Delta_M)$. Denote by $P\subseteq B(\mathscr{L}^2(N))$ the von Neumann algebra which is generated by elements of the form $xy^*$, where $x,y\in N$. Then $P$ can be made into a von Neumann bialgebra, the comultiplication $\Delta_P$ being uniquely determined by the fact that \[\Delta_P(xy^*)=\Delta_N(x)\Delta_N(y)^*\qquad \textrm{for all }x,y\in N.\] The von Neumann bialgebra $(P,\Delta_P)$ furthermore admits (not necessarily finite) left and right $\Delta_P$-invariant nsf weights (i.e.~ is a von Neumann algebraic quantum group in the terminology of \cite{Kus2}).
\end{Theorem}

\noindent The following theorem gives a concrete formula for the above invariant weights. We will use the notations introduced in the first three sections (see Notation \ref{NotWeight} and Notation \ref{NotA}).

\begin{Theorem}\label{TheoFormW} Let $(M,\Delta_M)$ be a compact Woronowicz algebra, $(N,\Delta_N)$ a right Galois co-object for $(M,\Delta_M)$, and $(P,\Delta_P)$ the von Neumann bialgebra introduced in Theorem \ref{TheoRefl}. Then, up to a positive scalar, the left invariant nsf weight $\varphi_{P}$ satisfies \[\widetilde{V}_{r,ij}\widetilde{V}_{s,kl}^* \in \mathscr{M}_{\varphi_P}\] with \[\varphi_{P}(\widetilde{V}_{r,ij}\widetilde{V}_{s,kl}^*) = \delta_{r,s}\delta_{i,k}\delta_{j,l} \frac{T_{r,j}}{A_{r,j}},\] while the right invariant nsf weight $\psi_{P}$ satisfies, again up to a positive scalar, \[\widetilde{W}_{r,ij}^*\widetilde{W}_{s,kl} \in \mathscr{M}_{\psi_P}\] and \[\psi_{P}(\widetilde{W}_{r,ij}^*\widetilde{W}_{s,kl}) = \delta_{r,s}\delta_{i,k}\delta_{j,l} \frac{T_{r,i}}{A_{r,i}}.\]

\end{Theorem}

\begin{proof} For the proof of the theorem, we have to explain first how the invariant weights $\varphi_{P}$ and $\psi_P$ can be obtained. This goes back to Theorem 4.8 of \cite{DeC3}.\\

\noindent Denote by $\nabla_{N,M}^{it}$ the following one-parametergroup of unitaries on $\mathscr{L}^2(N)$: \[\nabla_{N,M}^{it} = \nabla_{\widehat{N}}^{it}J_{\widehat{N}}\delta_{\widehat{N}}^{it}J_{\widehat{N}},\] where $\nabla_{\widehat{N}}$ is the modular operator associated the weight $\varphi_{\widehat{N}}$ on $\widehat{N}$. On basis vectors, this one-parametergroup is concretely given as \[\nabla_{N,M}^{it} e_{r,i}\otimes \overline{e_{r,j}} = A_{r,j}^{-it}T_{r,j}^{it}T_{r,i}^{-it} \, e_{r,i}\otimes \overline{e_{r,j}}.\] We can then implement on $N$ a one-parametergroup $\sigma_{t}^{N,M}$, determined by the formula \[\sigma_t^{N,M}(x) = \nabla_{N,M}^{it}x\nabla_{M}^{-it},\qquad \textrm{for all }x\in N,\] where $\nabla_M$ is the modular operator on $\mathscr{L}^2(M)$ associated to $\varphi_M$. One verifies that this is well-defined by using the commutation relation \[(\nabla_{M}^{it}\otimes 1)\widetilde{W}(\nabla_{N,M}^{-it}\otimes 1)= (1\otimes \nabla_{\widehat{N}}^{-it})\widetilde{W}(1\otimes \nabla_{\widehat{N}}^{it}\delta_{\widehat{N}}^{it}),\] proven in Proposition 3.20 of \cite{DeC3} (we remark that the $P$-operator introduced there coincides with $\nabla_{\widehat{N}}$, as the modular element of $(M,\Delta_M)$ is trivial). This commutation relation also immediately shows that \[\sigma_t^{N,M}(\widetilde{V}_{r,ij}) = T_{r,i}^{-it} T_{r,j}^{it}A_{r,j}^{-it} \widetilde{V}_{r,ij},\] using the notation from Lemma \ref{LemFormV}.1.\\

\noindent Now by construction (see the discussion preceding Lemma 4.4 in \cite{DeC3}), all elements $x\in N$ which are analytic with respect to $\sigma_t^{N,M}$ will lie in the space of square integrable elements of $\varphi_{P}$, by the formula \[\varphi_P(xx^*) = \varphi_M(\sigma_{-i/2}^{N,M}(x)^*\sigma_{-i/2}^{N,M}(x)).\] By polarity, we get for $x,y\in N$ analytic w.r.t.~ $\sigma_t^{N,M}$ that \[\varphi_P(xy^*) = \varphi_M(\sigma_{-i/2}^{N,M}(y)^*\sigma_{-i/2}^{N,M}(x)).\] Applying this to $x=\widetilde{V}_{r,ij}$ and $y=\widetilde{V}_{s,kl}$, and using the orthogonality relations between the $\widetilde{V}_{r,ij}$, we immediately get the first formula in the statement of the theorem.\\

\noindent For the second formula, we can use the expression $\psi_P = \varphi_P\circ R_P$, where $R_P$ was an involutory anti-automorphism on $P$ determined by the formula \[R_P(x) = J_{\widehat{N}}x^*J_{\widehat{N}}, \qquad \textrm{for all }x\in P\] (see Lemma 4.3 in \cite{DeC3}). In fact, the discussion before that lemma states that, for $x,y\in N$, we have $R_P(xy^*) = R_N(y)^*R_N(x)$, where \[R_N: N\rightarrow N^{\textrm{op}}: x\rightarrow J_{\widehat{M}}x^*J_{\widehat{N}}.\] By applying $R_N(\widetilde{V}_{r,ij})^*$ to $\xi_M$, we find that \[R_N(\widetilde{V}_{r,ij}) = \widetilde{W}_{r,ji}\] (see also the proof of Proposition \ref{PropComp2}). Applying then $\varphi_P\circ R_P$ to $\widetilde{W}_{r,ij}^*\widetilde{W}_{s,kl}$ and using the first part of the proof, the expression for $\psi_P$ as in the statement of the theorem follows.\\

\end{proof}

\noindent From the formulas in Theorem \ref{TheoFormW}, we can draw the following conclusions.

\begin{Prop}\label{PropFin} Let $(M,\Delta_M)$, $(N,\Delta_N)$ and $(P,\Delta_P)$ be as in the foregoing theorem. \begin{enumerate} \item If $(P,\Delta_P)$ is a compact Woronowicz algebra (that is, if $\varphi_{P}$ and $\psi_P$ are finite), then all $n_r< \infty$, i.e.~ all irreducible $(N,\Delta_N)$-corepresentations of $(M,\Delta_M)$ are finite-dimensional.
\item Conversely, if one of the irreducible $(N,\Delta_N)$-corepresentations for $(M,\Delta_M)$ is finite-dimensional, then they all are, and then $(P,\Delta_P)$ is a compact Woronowicz algebra.
\item If $(P,\Delta_P)$ is unimodular (that is, if $\varphi_P$ is a multiple of $\psi_P$), then there exist positive numbers $d_r$ such that $A_r = d_r^2T_r^2$ (where the $A_r$ were introduced in Notation \ref{NotA}).
\end{enumerate}
\end{Prop}

\noindent \emph{Remark:} If the condition in the third point is satisfied, then one could interpret $d_r$ as a (finite!) \emph{relative quantum dimension} of the irreducible $(N,\Delta_N)$-corepresentation corresponding to $r$, in analogy with the case of ordinary irreducible corepresentations (compare with Proposition \ref{PropDisc}). Here the relativity refers to one irreducible $(N,\Delta_N)$-corepresentation w.r.t.~ another, as $\delta_{\widehat{N}}$ is only determined up to a positive scalar. We note that we do not know of any particular examples where $(P,\Delta_P)$ is not unimodular, so it could well be that this condition is \emph{always} fulfilled.\\

\begin{proof} The third point follows immediately from the formulas in the previous theorem combined with Lemma \ref{LemFormV}, as there then exists a positive number $c>0$ such that  \[\frac{A_{r,j}}{A_{r,i}}= c \frac{T_{r,j}^2}{T_{r,i}^2}.\]

\noindent If $(P,\Delta_P)$ is moreover compact, then for any $r\in I_N$ and $0\leq i<n_r$, we get, by using the unitary of $\widetilde{W}_r$ (and the normality of $\psi_P$), that \begin{eqnarray*} \psi_P(1) &=& \psi_{P} (\sum_{i=0}^{n_r-1} \widetilde{W}_{r,ij}^*\widetilde{W}_{r,ij}) \\ &=& \sum_{i=0}^{n_r-1} \psi_P(\widetilde{W}_{r,ij}^*\widetilde{W}_{r,ij}) \\ &=& \frac{1}{d_r^2} \sum_{i=0}^{n_r-1} \frac{1}{T_{i,r}} <\infty.\end{eqnarray*} As the $T_{r,i}$ are summable, the final sum must necessarily be finite, i.e. $n_r<\infty$.\\

\noindent Finally, suppose that $(M,\Delta_M)$ has a finite-dimensional irreducible $(N,\Delta_N)$-corepresentation, say corresponding to the index value $r$. Then as $\psi_P(1) = \psi_{P} (\sum_{i=0}^{n_r-1} \widetilde{W}_{r,ij}^*\widetilde{W}_{r,ij})$, we see that $\psi_P$ is finite, and hence $(P,\Delta_P)$ is a compact Woronowicz algebra. By the second point, also all other irreducible $(N,\Delta_N)$-corepresentations of $(M,\Delta_M)$ are finite-dimensional.

\end{proof}

\noindent \emph{Remark:} In case the irreducible $(N,\Delta_N)$-corepresentations are finite-dimensional, the linear span of the $\widetilde{W}_{r,ij}^*$ generates inside $N$ a purely algebraic Galois co-object $\mathscr{N}$ for the Hopf algebra $\mathscr{A}$ inside $(M,\Delta_M)$. Conversely, if one starts with a Galois co-object for $\mathscr{A}$, satisfying some suitable relations with the $^*$-structure, we can in essence develop the whole theory so far in an algebraic way, and then \emph{necessarily} the reflection will correspond to a compact quantum group (this was essentially already observed in \cite{DeC4}, see also \cite{VDae1}). As it turns out, there \emph{do} exist interesting Galois co-objects which are of a non-algebraic type (see the final section), which was part of the motivation for writing this paper.\\

\section{Galois co-objects and projective corepresentations for compact Kac algebras}

\noindent In this short section, we show that when one deals with compact Kac algebras (see Definition \ref{DefCQG}), one is \emph{always} in the algebraic setup (see the remark at the end of the previous section).

\begin{Prop}\label{PropKac} Let $(N,\Delta_N)$ be a Galois co-object for a compact Kac algebra $(M,\Delta_M)$, and let $(P,\Delta_P)$ be the reflected von Neumann bialgebra as obtained in Theorem \ref{TheoRefl}. Then also $(P,\Delta_P)$ is a compact Kac algebra.\end{Prop}

\begin{proof} \noindent As we recalled in Theorem \ref{TheoGalob}, the modular element $\delta_{\widehat{N}}^{it}$ satisfies $\alpha_{\widehat{N}}(\delta_{\widehat{N}}^{it}) = \delta_{\widehat{N}}^{it}\otimes \delta_{\widehat{M}}^{it}$. However, for a compact Kac algebra, $\delta_{\widehat{M}} = 1$. By ergodicity of $\alpha_{\widehat{N}}$, we then conclude that we can take $\delta_{\widehat{N}} = 1$.\\

\noindent From Theorem \ref{TheoFormW} and the normality of $\psi_P$, we then find \begin{eqnarray*} \psi_P(1) &=& \psi_{P} (\sum_{i=0}^{n_r-1} \widetilde{W}_{r,ij}^*\widetilde{W}_{r,ij}) \\ &=& \sum_{i=0}^{n_r-1} \psi_P(\widetilde{W}_{r,ij}^*\widetilde{W}_{r,ij}) \\ &=& \sum_{i=0}^{n_r-1} T_{i,r} =1,\end{eqnarray*} so that $\psi_P$ is finite, and $(P,\Delta_P)$ thus a compact Woronowicz algebra.\\

\noindent But then $(P,\Delta_P)$ is in particular unimodular, so that the third point of Proposition \ref{PropFin} gives us that $T_r$ is a scalar matrix, and hence just $\frac{1}{n_r}$ times the unit matrix on $\mathscr{H}_r$. This shows that the one-parameter-group $\nabla_{N,M}^{it}$, which we introduced in the course of the proof of Theorem \ref{TheoFormW}, is trivial. As $\sigma_t^{\varphi_P}(xy^*) = \sigma_t^{N,M}(x)\sigma_t^{N,M}(y)^*$ for all $x,y\in N$ (which follows from the fact that $\sigma_t^{N,M}$ is actually the restriction to $N$ of the modular automorphism group of the balanced weight $\varphi_P\oplus \tau_M$ on $\left(\begin{array}{ll} P & N\\N^{\textrm{op}} & M\end{array}\right)$), we get that $\sigma_t^{\varphi_P}$ is trivial, and hence $\varphi_P$ is a trace. This concludes the proof.

\end{proof}

\noindent Combining the previous proposition with Theorem \ref{TheoPW} and Proposition \ref{PropFin}, we obtain the following corollary.

\begin{Cor}\label{CorKac} Let $(M,\Delta_M)$ be a compact Kac algebra. If $\mathscr{H}$ is a Hilbert space, and $\alpha: B(\mathscr{H})\rightarrow M\bar{\otimes} B(\mathscr{H})$ a coaction, then the following statements hold.
\begin{enumerate} \item The trace on $B(\mathscr{H})$ is $\alpha$-invariant.
\item If $\alpha$ is ergodic, then $\mathscr{H}$ is finite-dimensional.
\end{enumerate}
\end{Cor}

\noindent In particular, this says that the invariant state associated to an ergodic coaction of a compact Kac algebra on a type $I$-factor is tracial. Note that this is \emph{not} true for ergodic coactions of Kac algebras on \emph{arbitrary} von Neumann algebras (counterexamples can be found in \cite{Bic1}). It would be nice to have a more direct proof of the above corollary, but we were not able to produce one.\\

\section{Projective corepresentations of finite-dimensional Kac algebras}

\noindent In this section, we will briefly review what can be said concerning the situation of finite-dimensional compact Woronowicz (and hence Kac) algebras.

\begin{Prop} Let $(M,\Delta_M)$ be a finite-dimensional Kac algebra, and $(N,\Delta_N)$ a right Galois co-object for $(M,\Delta_M)$. Then $N$ is finite-dimensional with $\textrm{dim}(N)=\textrm{dim}(M)$, and $(N,\Delta_N)$ is cleft.\end{Prop}

\begin{proof} Choose $r\in I_N$. Then as the right coaction $\textrm{Ad}_R^{(r)}$ of $(M,\Delta_M)$ on $B(\mathscr{H}_r)$ is ergodic, it is easy to see \emph{immediately} that $n_r = \textrm{dim}(\mathscr{H}_r)$ is finite. Then take $r\in I_N$ fixed. As $N$ is the $\sigma$-weak closure of the set $\{\widetilde{W}_{r,ij}^{*}m\mid 0\leq i,j\leq n_r, m\in M\}$ by Lemma \ref{LemGen}, we see that $N$ is finite-dimensional. As $\widetilde{W}^*$ gives a unitary from $\mathscr{L}^2(N)\otimes \mathscr{L}^2(M)$ to $\mathscr{L}^2(N)\otimes \mathscr{L}^2(N)$, necessarily $\textrm{dim}(N)=\textrm{dim}(M)$.\\

\noindent Now disregarding the $^*$-structures, we get that $(N,\Delta_N)$ is a Galois co-object for the Hopf algebra $(M,\Delta_M)$. It is then well-known that $(N,\Delta_N)$ is cleft in this weaker form (see for example the remark following Corollary 3.2.4 in \cite{Sch1}). But this means that $N\cong M$ as right $M$-modules. As $M$ is a direct sum of matrix algebras, it is easy to see that we can in fact find a unitary $u:\mathscr{L}^2(N)\rightarrow \mathscr{L}^2(M)$ such that $uN = M$. Hence we may identify $\mathscr{L}^2(N)$ with $\mathscr{L}^2(M)$ and $N$ with $M$. We can then consider $\Omega = \Delta_N(1_M)$. By right linearity of $\Delta_N$, we then get $\Delta_N(x) = \Omega\Delta_M(x)$ for all $x\in N$, and by coassociativity of $\Delta_N$ we have that $\Omega$ satisfies the 2-cocycle relation. Hence $(N,\Delta_N)$ is cleft.

\end{proof}

\noindent \emph{Remark:} Finite Galois co-objects (in the operator algebra context) have also been dealt with in the papers \cite{Eno1}, \cite{Vai1} and, in a more general setting, \cite{Izu1}.\\

\noindent For later purposes, we also introduce the following definition of a \emph{non-degenerate} 2-cocycle (see Definition \ref{DefCoc} for the general notion of a unitary 2-cocycle).

\begin{Def}\label{DefNonDeg} Let $(M,\Delta_M)$ be a finite-dimensional compact Kac algebra, and $\Omega$ a unitary 2-cocycle for $(M,\Delta_M)$. We call $\Omega$ \emph{non-degenerate} if the associated Galois object $\widehat{N}$ is a (finite-dimensional) type $I$-factor.\end{Def}

\noindent This terminology was introduced in \cite{Eti1}. One observes that, as $(\widehat{N},\alpha_{\widehat{N}})$ is then in fact also a projective (right) corepresentation for $(\widehat{M},\Delta_{\widehat{M}})$, we can create from it a Galois co-object for $(\widehat{M},\Delta_{\widehat{M}})$, which will then necessarily also be cleft. If we denote by $\widehat{\Omega}\in \widehat{M}\otimes \widehat{M}$ an implementing unitary 2-cocycle, then $\widehat{\Omega}$ turns out to be non-degenerate again, and the correspondence $\lbrack \Omega \rbrack \rightarrow \lbrack \widehat{\Omega}\rbrack$ between cohomology classes of non-degenerate 2-cocycles of resp.~ $(M,\Delta_M)$ and $(\widehat{M},\Delta_{\widehat{M}})$ is a bijection. (For the proof of this result, we refer again to \cite{Eti1}.) To give a simple example, consider a finite abelian group $G$. Then the bicharacter on $\widehat{G}\times G$ given by evaluation gives a non-degenerate 2-cocycle function $\Omega$ on $G\times \widehat{G}$ by the formula $((g,\chi),(h,\chi')) \rightarrow \chi(h)$, and its dual is simply the same construction applied to the evaluation bicharacter on $\widehat{G}\times G$.\\

\section{Projective corepresentations for cocommutative Kac algebras}

\noindent As a second special case, we will consider compact Woronowicz algebras $(M,\Delta_M)$ which have a cocommutative coproduct: $\Delta^{\textrm{op}}_M = \Delta_M$. It is not so difficult to show that $(M,\Delta_M)$ is then Kac, and in fact that $M = \mathscr{L}(\Gamma)$ for some (countable) discrete group $\Gamma$, the coproduct being given by \[\Delta_M(\lambda_g) = \lambda_g\otimes \lambda_g\qquad \textrm{for all }g\in \Gamma,\] where the $\lambda_g$ denote the standard unitary generators in the group von Neumann algebra $\mathscr{L}(\Gamma)$. We will in the following denote $\mathscr{L}(\Gamma)$ as shorthand for $(\mathscr{L}(\Gamma),\Delta_{\mathscr{L}(\Gamma)})$, and we denote the invariant trace by $\tau$. The dual discrete Woronowicz algebra $(\widehat{M},\Delta_{\widehat{M}})$ is then simply the function space $l^{\infty}(\Gamma)$, equipped with the coproduct \[\Delta_{\widehat{M}}:l^{\infty}(\Gamma)\rightarrow l^{\infty}(\Gamma)\bar{\otimes} l^{\infty}(\Gamma)\cong l^{\infty}(\Gamma\times \Gamma)\] such that\[\Delta_{\widehat{M}}(f)(g,h) = f(gh) \qquad \textrm{for all }g,h\in \Gamma.\]

\noindent We will in the following write $\mathscr{L}^2(\mathscr{L}(\Gamma)) = l^2(\Gamma)$ of course, and then, with $\delta_g$ being the Dirac function at the point $g\in \Gamma$, we have \[\Lambda_{\mathscr{L}(\Gamma)}(\lambda_g) = \delta_g.\]

\noindent As group von Neumann algebras are in particular Kac algebras, we know from Corollary \ref{CorKac} that they can only act ergodically on type I factors which are of finite type. Let us give a more immediate proof of this fact in this particular case.

\begin{Lem} Let $\Gamma$ be a discrete group. Let $B$ be a von Neumann algebra, and suppose that we have given an ergodic coaction \[\alpha: B\rightarrow \mathscr{L}(\Gamma)\otimes B.\] If we denote by $\phi_B$ the unique $\alpha$-invariant state on $B$, then $\phi_B$ is a trace.
\end{Lem}

\begin{proof} For each $g\in \Gamma$, denote $B_g = \{x\in B\mid \alpha(x) = \lambda_g\otimes x\}$. As $\alpha$ is ergodic, it is easily seen that each $B_g$ is either zero- or one-dimensional. It is further immediate that $B_g\cdot B_h \subseteq B_{gh}$ for all $g,h\in \Gamma$, and that $B_g^* = B_{g^{-1}}$. Therefore, whenever $B_g$ is not zero-dimensional, we may assume that $B_g = \mathbb{C}u_g$ with $u_g$ a unitary. We may moreover assume that $u_g^* = u_{g^{-1}}$. When $B_g =0$, we will denote $u_g=0$.\\

\noindent We claim that the linear span of the $u_g$ is $\sigma$-weakly dense in $B$. Indeed, if this was not the case, then we could find a nonzero $x\in B$ with $\phi_B(xu_g)=0$ for all $g\in \Gamma$. But as $\phi_B = (\tau\otimes \iota)\alpha$ by definition, this would imply that \[(\tau\otimes \iota)(\alpha(x)(\lambda_g\otimes 1))u_g =0\qquad \textrm{for all }g\in \Gamma.\] Now an easy computation shows that for all $g\in \Gamma$, we have \[(\tau\otimes \iota)(\alpha(x)(\lambda_g\otimes 1)) \in B_g^*.\] Hence we in fact have \[(\tau\otimes \iota)(\alpha(x)(\lambda_g\otimes 1)) =0\qquad \textrm{for all }g\in \Gamma.\] This implies $\alpha(x)=0$, and so $x=0$, which is a contradiction.\\

\noindent It is now enough to prove that $\phi_B(u_gu_h)=\phi_B(u_hu_g)$ for all $g,h\in \Gamma$. But the left hand side is a multiple of $\phi_B(u_{gh})$, which is zero in case $g\neq h^{-1}$. Similarly, the right hand side is zero in case $g\neq h^{-1}$. As both sides equal $1$ when $g=h^{-1}$, we are done.
\end{proof}

\noindent \emph{Remark:} General coactions of group von Neumann algebras (or rather, of the associated group C$^*$-algebras), have been studied in the theory of \emph{Fell bundles} over groups (see for example \cite{Qui1}). The intuitive connection between these notions is essentially contained the above proof.\\

\begin{Cor} Let $\Gamma$ be a discrete group, $\mathscr{H}$ a Hilbert space, and \[\alpha:B(\mathscr{H})\rightarrow \mathscr{L}(\Gamma)\bar{\otimes}B(\mathscr{H})\] an ergodic coaction of $\mathscr{L}(\Gamma)$ on $B(\mathscr{H})$. Then the following statements hold.

\begin{enumerate} \item The dimension of $\mathscr{H}$ is finite.
\item There exists a finite subgroup $H$ of $\Gamma$ such that \[\alpha(B(\mathscr{H}))\subseteq \mathscr{L}(H)\otimes B(\mathscr{H}) ,\] and such that, denoting by $\beta$ the coaction $\alpha$ with range restricted to $\mathscr{L}(H)\otimes B(\mathscr{H})$, the couple $(B(\mathscr{H}),\beta)$ is a (left) Galois object for $\mathscr{L}(H)$.
\end{enumerate}
\end{Cor}

\noindent \emph{Remark:} The notion of a Galois object was introduced in the second section. In the finite-dimensional setting, it may be defined as follows: let $A$ be a finite-dimensional Hopf $^*$-algebra with a left coaction $\beta$ on a finite-dimensional $^*$-algebra $B$. Then $(B,\beta)$ is called a left Galois object for $A$ if the map \[B\otimes B\rightarrow A\otimes B: x\otimes y\rightarrow \beta(x)(1\otimes y)\] is a bijection.

\begin{proof} By the previous lemma, we know that $B(\mathscr{H})$ admits a tracial state. Hence $\mathscr{H}$ must be finite.\\

\noindent As for the second point, this is rather a corollary of the proof of the previous proposition. For, using the notation introduced there, denote by $H$ the set of elements $g$ in $\Gamma$ for which $u_g\neq 0$. As the $u_g$ are orthogonal to each other with respect to the $\alpha$-invariant state, we must have that $H$ is finite. As $u_g\cdot u_h$ is a non-zero multiple of $u_{gh}$, and $u_g^* = u_{g^{-1}}$, we must have that $H$ is a finite group. It is then immediate that indeed $\alpha(B(\mathscr{H}))\subseteq \mathscr{L}(H)\otimes B(\mathscr{H}) $.\\

\noindent The coaction $\beta$ of $\mathscr{L}(H)$ on $B(\mathscr{H})$ is then clearly also an ergodic action, with the ordinary (normalized) trace $\textrm{tr}$ as its invariant state. This implies that the map \[\mathscr{L}^2(B(\mathscr{H}),\textrm{tr})\otimes \mathscr{L}^2(B(\mathscr{H}),\textrm{tr})\rightarrow (\mathscr{L}^2(\mathscr{L}(H)),\tau) \otimes\mathscr{L}^2(B(\mathscr{H}),\textrm{tr}): x\otimes y\rightarrow \beta(x)(1\otimes y)\] is isometric and thus injective. As the $u_g$ with $g\in H$ form an orthonormal basis of $B(\mathscr{H})$, we have that the order of $H$ equals the square of the dimension of $\mathscr{H}$. Hence a comparison of dimensions shows that the above map is also surjective, which proves that $(B(\mathscr{H}),\beta)$ is a left Galois object for $\mathscr{L}(H)$.

\end{proof}

\begin{Prop}\label{PropDG} Let $\Gamma$ be a discrete group, and let $(N,\Delta_N)$ be a right Galois co-object for $\mathscr{L}(\Gamma)$. Then there exists a finite subgroup $H\subseteq \Gamma$ and a non-degenerate 2-cocycle $\Omega$ for $\mathscr{L}(H)$, such that $(N,\Delta_N)$ is isomorphic to the cleft Galois co-object induced by $\Omega$ (considered as a 2-cocycle for $\mathscr{L}(\Gamma)$).\\
\end{Prop}

\begin{proof} Let $(N,\Delta_N)$ be a right Galois co-object for $\mathscr{L}(\Gamma)$. Using the terminology introduced in Definition \ref{DefImpl}, choose an irreducible $\lbrack(N,\Delta_N)\rbrack$-corepresentation \[\alpha: B(\mathscr{H})\rightarrow \mathscr{L}(\Gamma)\otimes B(\mathscr{H})\] of $\mathscr{L}(\Gamma)$ on a Hilbert space $\mathscr{H}$ (for example, one of the $\textrm{Ad}_L^{(r)}$, see the end of the third section). By the previous proposition, we know that $\mathscr{H}$ is finite-dimensional, and that we can choose a minimal finite subgroup $H\subseteq \Gamma$ for which \[\alpha(B(\mathscr{H}))\subseteq \mathscr{L}(H)\otimes B(\mathscr{H}).\] We moreover know that the corresponding coaction $\beta$ of $\mathscr{L}(H)$ on $B(\mathscr{H})$ is then a Galois object. This means that the Galois co-object $(N_H,\Delta_{N_H})$ which is associated to $\beta$ (as a projective corepresentation) may be taken to be equal to $(\mathscr{L}(H),\Omega \Delta_{\mathscr{L}(H)}(\,\cdot\,))$, where $\Omega\in \mathscr{L}(H)\otimes \mathscr{L}(H)$ is a \emph{non-degenerate} unitary 2-cocycle (see the final remarks of the previous section). Denote further $(\widetilde{N},\Delta_{\widetilde{N}}) := (\mathscr{L}(\Gamma),\Omega\Delta_{\mathscr{L}(\Gamma)}(\,\cdot\,))$, which is a cleft Galois co-object for $\mathscr{L}(\Gamma)$.\\

\noindent Let then $\mathcal{G}$ be a projective $(N_H,\Delta_{N_H})$-corepresentation which implements $\beta$. As $N_H= \mathscr{L}(H) \subseteq \widetilde{N}=\mathscr{L}(\Gamma)$, we may interpret $\mathcal{G}$ to be an element of $\widetilde{N}\otimes B(\mathscr{H})$. It is trivial to see that $\mathcal{G}$ is then an $(\widetilde{N},\Delta_{\widetilde{N}})$-corepresentation which implements $\alpha$. Therefore $(N,\Delta_N)$ is isomorphic to $(\widetilde{N},\Delta_{\widetilde{N}})$ as a right Galois co-object for $\mathscr{L}(\Gamma)$ (see the remark after Definition \ref{DefImpl}), which proves the proposition.
\end{proof}

\begin{Cor}\label{CorTriv} If $\Gamma$ is a torsionless discrete group, or more generally, a group with no finite subgroups of square order, then any Galois co-object for $(\mathscr{L}(\Gamma),\Delta_{\mathscr{L}(\Gamma)})$ is isomorphic to $(\mathscr{L}(\Gamma),\Delta_{\mathscr{L}(\Gamma)})$ as a right Galois co-object.\\

\noindent In particular, any unitary 2-cocycle for $\mathscr{L}(\Gamma)$ is then a 2-coboundary (see the remark after Example \ref{ExaCocy}).\end{Cor}

\begin{proof} The statement concerning torsionless discrete groups is of course an immediate consequence of the previous proposition. As for the statement concerning the case when there are no finite subgroups of square order, observe that if $K$ is any finite group for which some $B(\mathscr{H})$ allows a Galois object structure for $l^{\infty}(K)$, then necessarily $|K| = \textrm{dim}(H)^2$.
\end{proof}

\noindent \emph{Remarks:} \begin{enumerate} \item For finite groups, Proposition \ref{PropDG} was proven in \cite{Mov1} (see also \cite{Eti2}).
\item In \cite{IPV1}, it is shown that for \emph{any} group $\Gamma$, all \emph{quasi-symmetric} 2-cocycles for $\mathscr{L}(\Gamma)$, i.e. cocycles which also satisfy $\Sigma \Omega\Sigma = \eta\Omega$ for some $\eta\in S^1$, are coboundaries. (The authors weaken this to allow unitaries satisfying the 2-cocycle condition up to a scalar, but it is possible to show that, in any compact Woronowicz algebra, such unitaries are \emph{automatically} 2-cocycles).
\item One can also easily describe the Galois objects dual to the Galois co-objects appearing in Proposition \ref{PropDG}. Namely, if we have given a discrete group $\Gamma$, a finite subgroup $H$ and a non-degenerate 2-cocycle $\Omega$ for $\mathscr{L}(H)$, let $\gamma:B(\mathscr{H})\rightarrow B(\mathscr{H})\otimes l^{\infty}(H)$ be the Galois object for $l^{\infty}(H)$ dual to the Galois co-object associated with $\Omega$. Then the dual of the Galois co-object for $\mathscr{L}(\Gamma)$ associated to $(H,\Omega)$ is the \emph{induction} of $\gamma$ to $\Gamma$. The underlying von Neumann algebra of this construction consists of the set of all elements $x\in B(\mathscr{H})\otimes l^{\infty}(\Gamma)$ for which \[(\gamma\otimes \iota)(x) = (\iota\otimes \beta_{l^{\infty}(H)})(x),\] where $\beta_{l^{\infty}(H)}$ is the coaction associated to the left translation action of $H$ on $\Gamma$. The right coaction $\alpha$ of $l^{\infty}(\Gamma)$ on this von Neumann algebra is then simply given by right translation, i.e.~ $\alpha(x) := (\iota\otimes\Delta_{l^{\infty}(\Gamma)})(x)$.\\
\end{enumerate}

\noindent The projective corepresentations associated to the Galois co-objects for $\mathscr{L}(\Gamma)$ can be determined as follows.

\begin{Prop} Let $\Gamma$ be a discrete group with a finite subgroup $H$. Let $\Omega$ be a non-degenerate 2-cocycle for $\mathscr{L}(H)$, and let $\mathcal{G} \in \mathscr{L}(H)\otimes B(\mathscr{H})$ be the associated irreducible $\Omega$-corepresentation on some Hilbert space $\mathscr{H}$.\\

\noindent Then with $(N,\Delta_N)$ the cleft Galois co-object for $\mathscr{L}(\Gamma)$ associated to $\Omega \in \mathscr{L}(\Gamma)\overline{\otimes} \mathscr{L}(\Gamma)$, we have $I_N\cong H\setminus\Gamma$, and a maximal set of irreducible non-isomorphic $(N,\Delta_N)$-corepresentations is given by the set \[\mathcal{G}_{Hg} := \mathcal{G}(\lambda_{s(Hg)}\otimes 1) \in \mathscr{L}(\Gamma)\otimes B(\mathscr{H}),\] where $s: H\setminus \Gamma \rightarrow \Gamma$ is a fixed section for $\Gamma \rightarrow H\setminus \Gamma$.

\end{Prop}

\noindent \emph{Remark:} As $\Omega$ is assumed to be non-degenerate for $\mathscr{L}(H)$, the unitary $\mathcal{G}$ is indeed the \emph{unique} $\Omega$-corepresentation for $\mathscr{L}(H)$, up to isomorphism.\\

\begin{proof} It is immediately seen that the right regular $(N,\Delta_N)$-corepresentation for $\mathscr{L}(\Gamma)$ equals $\widetilde{V} = \Omega V$, while the left regular $(N^{\textrm{op}},\Delta_{N^{\textrm{op}}})$-corepresentation equals $\widetilde{W}= \Omega_H^* W$. For $g\in \Gamma$, let $\delta_{Hg}\in B(l^2(\Gamma))$ be the indicator function for the coset $Hg$. Clearly, $\delta_{Hg}$ commutes pointwise with $\mathscr{L}(H)$. Using then the definition of $\widehat{N}$ as a fixed point set (see Proposition \ref{PropDen2}), it is easy to check that $\delta_{Hg}\in \widehat{N}$. Using the second definition of $\widehat{N}$ as the closure of the right leg of $\widetilde{W}$ (see again Proposition \ref{PropDen2}), we get that in fact $\delta_{Hg}\in \mathscr{Z}(\widehat{N})$, the center of $\widehat{N}$.\\

\noindent We claim now that $I_N \cong H\setminus \Gamma$, and that the $\widetilde{V}_{Hg} := \widetilde{V}( \delta_{Hg}\otimes 1)$ are the irreducible components of $\widetilde{V}$. Indeed, as $\widetilde{V} = \oplus_{gH} \widetilde{V}_{Hg}$, it is enough to show that each of the sets $\{(\iota\otimes \omega)(\widetilde{V}_{Hg})\mid \omega\in \mathscr{L}(\Gamma)_*\}$ is a type $I$-factor. But denoting $V_H = \sum_{h\in H} \delta_h\otimes \lambda_h$, an easy computation shows that \[\widetilde{V}_{Hg} = (\rho_g\otimes 1)\Omega V_H (\rho_g^*\otimes \lambda_g),\] where $\rho_g$ is a right translation operator on $l^2(\Gamma)$, i.e.~ $\rho_g\delta_k = \delta_{kg}$ for $g,k\in \Gamma$. As we assumed that $\Omega$ is a non-degenerate 2-cocycle, we know that $\{(\iota\otimes \omega)(\Omega V_{H})\mid \omega\in \mathscr{L}(\Gamma)_*\}$ is a type $I$-factor. This proves the claim.\\

\noindent Now the irreducible $\Omega$-corepresentation $\Sigma \widetilde{V}_{Hg}\Sigma$ of $\mathscr{L}(\Gamma)$ is immediately seen to be isomorphic to the $\Omega$-corepresentation $\mathcal{G}_{gH}$ in the statement of the proposition, while $\widetilde{V}_H$ is isomorphic to $\mathcal{G}$ as an $\Omega$-corepresentation for $\mathscr{L}(H)$. This then finishes the proof.

\end{proof}

\noindent \emph{Remarks:}\begin{enumerate}\item  With the help of this proposition, one can show that if $H$ and $K$ are two finite subgroups of $\Gamma$, with respective non-degenerate 2-cocycles $\Omega_H$ and $\Omega_K$, then the associated Galois co-objects $(N,\Delta_{N})$ and $(\widetilde{N},\Delta_{\widetilde{N}})$ for $\mathscr{L}(\Gamma)$ are isomorphic iff there exists $g\in \Gamma$ and a unitary $u\in \mathscr{L}(H)$ with $g^{-1}Hg=K$ and \[(\lambda_g\otimes\lambda_g)\Omega_2(\lambda_g^*\otimes \lambda_g^*) = (u^*\otimes u^*)\Omega_1\Delta_{\mathscr{L}(H)}(u).\]

\item One can also be more specific on when the projective corepresentations associated to two given $(N,\Delta_N)$-corepresentations as above are actually isomorphic. Namely, using the notation as in the statement of the proposition, let $\alpha_g$ be the coaction of $\mathscr{L}(\Gamma)$ on $B(\mathscr{H})$ implemented by $\mathcal{G}_{Hg}$. We may assume that $\alpha_e$ is the `extension' of the coaction $\beta$ of $\mathscr{L}(H)$ on $B(\mathscr{H})$ implemented by $\mathcal{G}$. Then we have $\alpha_g\cong \alpha_e$ iff $gHg^{-1}=H$ and $\Omega$ is coboundary equivalent to $(\lambda_g\otimes \lambda_g)\Omega(\lambda_g^*\otimes\lambda_g^*)$ (inside $\mathscr{L}(H)$).

\end{enumerate}

\section{A projective representation for $SU_q(2)$}

\noindent In this section, we want to consider one special and non-trivial example of a projective representation of the compact quantum group $SU_q(2)$. This projective representation will be nothing else but (a completion of) its action on the standard Podle\`{s} sphere.\\

\noindent Let us first recall the definition of $SU_q(2)$ on the von Neumann algebra level. \emph{For the rest of this section, we fix a number $0<q<1$.}

\begin{Def}\label{DefSUq2} Denote $\mathscr{L}^2(SU_q(2))= l^2(\mathbb{N}) \otimes \overline{l^2(\mathbb{N})} \otimes l^2(\mathbb{Z})$. Consider on it the operators \[a=\sum_{k\in \mathbb{N}_0} \sqrt{1-q^{2k}}\,e_{k-1,k}\otimes 1 \otimes 1,\] \[b= (\sum_{k\in \mathbb{N}} q^k\,e_{kk})\otimes 1\otimes S,\] where $S$ denotes the forward bilateral shift.\\

\noindent Then the compact Woronowicz algebra $(\mathscr{L}^{\infty}(SU_q(2)),\Delta_+)$ consists of the von Neumann algebra \[\mathscr{L}^{\infty}(SU_q(2)) = B(l^2(I_+))\bar{\otimes}1\bar{\otimes}\mathscr{L}(\mathbb{Z}) \subseteq B(\mathscr{L}^2(SU_q(2))),\] equipped with the unique unital normal $^*$-homomorphism \[\Delta_+: \mathscr{L}^{\infty}(SU_q(2))\rightarrow \mathscr{L}^{\infty}(SU_q(2))\bar{\otimes}\mathscr{L}^{\infty}(SU_q(2))\] which satisfies
\[\left\{\begin{array}{l} \Delta_+(a) = a\otimes a - q b^* \otimes b \\ \Delta_+(b) = b\otimes a + a^*\otimes b.\end{array}\right.\]

\noindent Its invariant state $\varphi_+$ is given by the formula \[\varphi_+(e_{ij}\otimes 1\otimes S^k) = \delta_{i,j} \delta_{k,0}(1-q^2)q^{2k},\qquad \textrm{for all }i,j\in \mathbb{N},k\in \mathbb{Z},\] and we may identify $\mathscr{L}^2(SU_q(2))$ with the Hilbert space in the GNS-construction for $\varphi_+$ by putting $\xi_M = \sqrt{1-q^2}\sum_{i\in \mathbb{N}} q^{i}\,e_i\otimes \overline{e_i}\otimes e_0$.
\end{Def}

\noindent The definition of the standard Podle\`{s} sphere and the associated action of $SU_q(2)$ takes the following form on the von Neumann algebraic level.

\begin{Def}\label{DefPod} Define $\mathscr{L}^{\infty}(S_{q0}^2)$ to be the von Neumann algebra inside $\mathscr{L}^{\infty}(SU_q(2))$ generated by $X=qb^*a$ and $Z=b^*b$. Then $\Delta_+$ restricts to a left (ergodic) coaction $\alpha$ of $\mathscr{L}^{\infty}(SU_q(2))$ on $\mathscr{L}^{\infty}(S_{q0}^2)$. We say that this coaction corresponds to `the action of $SU_q(2)$ on the standard Podle\`{s} sphere'.
\end{Def}

\noindent One can show that $\mathscr{L}^{\infty}(S_{q0}^2)$ may be identified with the von Neumann algebra $B(l^2(\mathbb{N}))$, in such a way that \[X\rightarrow \sum_{k\in \mathbb{N}_0} q^k \sqrt{1-q^{2k}}\; e_{k-1,k}\]\[Z\rightarrow \sum_{k\in \mathbb{N}} q^{2k}\, e_{kk}.\] Under this correspondence, the $\alpha$-invariant state on $\mathscr{L}^{\infty}(S_{q0}^2) = B(l^2(\mathbb{N}))$ equals \[\phi_{\alpha}(e_{ij}) = \delta_{ij}(1-q^2)q^{2i}.\]

\noindent In the terminology of the present paper, the coaction $\alpha$ is thus an irreducible projective representation of $SU_q(2)$ on an infinite dimensional Hilbert space. In \cite{DeC2}, we computed the associated Galois co-object. To introduce it, let us first recall some notations from $q$-analysis (see \cite{Gas1}).

\begin{Not} For $n\in \mathbb{N}\cup\{\infty\}$ and $a\in \mathbb{C}$, we denote \[(a;q)_{n} = \overset{n-1}{\underset{k=0}{\prod}} (1-q^ka),\] which determines analytic functions in the variable $a$ with no zeroes in the open unit disc.\\

\noindent For $n\in \mathbb{N}$ and $a\in \mathbb{C}$, we denote by $p_n(x;a,0\mid q)$ the Wall polynomial of degree $n$ with parameter value $a$; so \[p_n(x;a,0\mid q) = \,_2\varphi_1(q^{-n},0;qa\mid q,qx),\] where $_2\varphi_1$ denotes Heine's basic hypergeometric function.
\end{Not}

\begin{Prop}\label{PropGalCobSU} Let $\mathscr{L}^2(N) = l^2(\mathbb{Z})\otimes \overline{l^2(\mathbb{Z})}\otimes l^2(\mathbb{Z})$. Denote by $v$ the operator $S^*\otimes1\otimes 1$, where $S$ denotes the forward bilateral shift, and by $L_{0+}$ the operator such that \[L_{0+}(e_n \otimes \overline{e_{m}}\otimes e_k) = (q^{2n+2};q^2)_{\infty}^{1/2}\,e_n\otimes \overline{e_m}\otimes e_k,\] so $L_{0+}=u(q^2b^*b;q^2)_{\infty}^{1/2}$ with $u$ the canonical isometric inclusion of $l^2(\mathbb{N})\otimes \overline{l^2(\mathbb{N})}\otimes l^2(\mathbb{Z})$ into $l^2(\mathbb{Z})\otimes \overline{l^2(\mathbb{Z})}\otimes l^2(\mathbb{Z})$. Denote $N$ for the $\sigma$-weak closure of the right $\mathscr{L}^{\infty}(SU_q(2))$-module generated by the elements $v^nL_{0+}$, where $n\in \mathbb{Z}$. (One easily shows that $N$ equals the set $B(l^2(\mathbb{N}),l^2(\mathbb{Z}))\bar{\otimes} 1\otimes \mathscr{L}(\mathbb{Z})$.)\\

\noindent Then there exists a unique Galois co-object structure $(N,\Delta_N)$ on $N$ for which \[\Delta_N(v^nL_{0+}) = (v^n\otimes v^n)\cdot (\sum_{p=0}^{\infty} (q^2;q^2)_p^{-1} \;v^pL_{0+}b^p\otimes v^pL_{0+}(-qb^*)^p),\] the right hand side converging in norm.\\

\noindent The coaction $\alpha$ of $\mathscr{L}^{\infty}(SU_q(2))$ on $\mathscr{L}^{\infty}(S_{q0}^2)$ is then an irreducible $\lbrack (N,\Delta_N)\rbrack$-corepresentation, and an associated implementing $(N,\Delta_N)$-corepresentation $\mathcal{G}$ is determined by the following formula: denoting $\mathcal{G}= \sum_{s,t\in \mathbb{N}} \mathcal{G}_{ts}\otimes e_{ts} \in N\bar{\otimes}B(l^2(\mathbb{N}))$, we have, for $0\leq t\leq s$, that \[\mathcal{G}_{ts} = q^{t(t-s)}\left(\frac{(q^2;q^2)_s}{(q^2;q^2)_t}\right)^{1/2} (q^2;q^2)_{s-t}^{-1} \cdot v^{s+t} L_{0+} b^{s-t}\cdot  p_t(b^*b;q^{2(s-t)},0\mid q^2),\] while for $0\leq s \leq t$, we have \[\mathcal{G}_{ts} = q^{s(s-t)}\left(\frac{(q^2;q^2)_t}{(q^2;q^2)_s}\right)^{1/2} (q^2;q^2)_{t-s}^{-1}\cdot v^{t+s} L_{0+} (-qb^*)^{t-s}  \cdot p_s(b^*b;q^{2(t-s)},0\mid q^2).\]
\end{Prop}

\noindent \emph{For the rest of this section, we take $(M,\Delta_M)$ to be $(\mathscr{L}^{\infty}(SU_q(2)),\Delta_{\mathscr{L}^{\infty}(SU_q(2))})$, and we fix the right Galois co-object $(N,\Delta_N)$ for $(M,\Delta_M)$ introduced above. We then also keep using the notations introduced above, as well as those from the first four sections.}\\

\noindent Our aim now is to describe in more detail the structure of the Galois co-object $(N,\Delta_N)$. In particular, we want to find a complete set of irreducible $(N,\Delta_N)$-corepresentations. This is in fact not so difficult.

\begin{Prop} Let $(N,\Delta_N)$ be the Galois co-object and $\mathcal{G}$ the $(N,\Delta_N)$-corepresentation introduced in Proposition \ref{PropGalCobSU}.\\

\noindent Then the set of unitaries \[\mathcal{G}^{(n)} := (v^n\otimes 1)\mathcal{G},\qquad n\in\mathbb{Z}\] forms a complete set of irreducible $(N,\Delta_N)$-corepresentations for $\mathscr{L}^{\infty}(SU_q(2))$.\\

\noindent In particular, the set $I_N$ of isomorphism classes of irreducible $(N,\Delta_N)$-corepresentations can be naturally identified with $\mathbb{Z}$.
\end{Prop}

\begin{proof} It is trivial to see that the $\mathcal{G}^{(n)}$ are indeed irreducible $(N,\Delta_N)$-corepresentations, by the group-like property of $v$. We then only need to show that the $\mathcal{G}^{(n)}$ are mutually non-isomorphic and have $\sigma$-weakly dense linear span in $N$.\\

\noindent We first prove that all $\mathcal{G}^{(n)}$ are mutually non-isomorphic. As an isomorphism between $\mathcal{G}^{(n)}$ and $\mathcal{G}^{(m)}$ would induce an isomorphism between $\mathcal{G}^{(0)}$ and $\mathcal{G}^{(m-n)}$, it is sufficient to show that $\mathcal{G}=\mathcal{G}^{(0)}$ is not isomorphic to $\mathcal{G}^{(n)}$ for $n\neq 0$. But for this, it is in turn sufficient to show that $L_{0+}=\mathcal{G}_{00}$ is orthogonal to all $\mathcal{G}^{(n)}_{ts}$, by the orthogonality relations in Lemma \ref{LemFormV}.2 (and Theorem \ref{TheoPW}.2). \\

\noindent Now we remark that $\varphi_+$ satisfies the property that $\varphi_+(a^mb^k(b^*)^l) = 0$ whenever $m\neq 0$ and $k\neq l$, and likewise with $a$ replaced by $a^*$. Moreover, one easily computes that the commutation relation $v^*L_{0+} = L_{0+}a^*$ holds. Using then the concrete form for the $\mathcal{G}_{rs}$ in the previous Proposition, it is easy to see that \[\varphi_+(L_{0+}^* \mathcal{G}_{ts}^{(n)}) \neq 0 \Rightarrow s=t\textrm{ and }n+2t=0.\] Hence the only thing left to do is to prove that $L_{0+}$ is orthogonal to $\mathcal{G}_{tt}^{(-2t)}$ for $t\in \mathbb{Z}_0$. But suppose this were not so. Then as $\mathcal{G}^{(-2t)}$ and $\mathcal{G}^{(0)}$ are irreducible, we would necessarily get that they are isomorphic, again by the orthogonality relations. As the inner product of $L_{0+}$ with all $\mathcal{G}_{rs}^{(-2t)}$ except $r,s = t$ is zero, this would then imply that $L_{0+}$ must be a scalar multiple of $\mathcal{G}_{tt}^{(-2t)}$, which is equivalent with saying that $p_{t}(b^*b;1,0\mid q^2)$ is a scalar multiple of the unit. As the spectrum of $b^*b$ is not finite, and $p_t(x;1,0\mid q^2)$ is a non-constant polynomial in $x$, we obtain a contradiction. Hence the $\mathcal{G}^{(n)}$ are mutually non-isomorphic.\\

\noindent We end by showing that the $\mathcal{G}_{ts}^{(n)}$ have a $\sigma$-weakly dense linear span in $N$. Consider, for $k,t\in \mathbb{N}$, the element $\mathcal{G}_{t,t+k}^{(-2t-k)}$. Then, up to a non-zero constant, this equals the element $L_{0+}b^kp_t(b^*b;q^{2k},0\mid q^2)$. As the $p_t(x;q^{2k},0\mid q^2)$ are polynomials of degree $t$, we conclude that the linear span of the $\mathcal{G}_{t,t+k}^{(-2t-k)}$ contains all elements of the form $L_{0+}b^{k+t}(b^*)^t$. A similar argument shows that the $\mathcal{G}_{s+k,s}^{(-2s-k)}$ contain all elements of the form $L_{0+}b^s(b^*)^{s+k}$. Hence the linear span of all $\mathcal{G}_{ts}^{(n)}$ contains all elements of the form $L_{0+}b^s(b^*)^t$ for $s,t\in \mathbb{N}$. As this linear span is closed under left multiplication with powers of $v$, we conclude that this linear span contains all elements of the form $v^nL_{0+}b^s(b^*)^t$ where $n\in \mathbb{Z},s,t\in \mathbb{N}$. As we can $\sigma$-weakly approximate elements of the form $e_{rr}\otimes 1\otimes S^k$ by elements in the linear span of the $b^s(b^*)^t$, it follows immediately that the $\sigma$-weak closure of the linear span of the $\mathcal{G}_{ts}^{(n)}$ indeed equals $N= B(l^2(\mathbb{N}),l^2(\mathbb{Z}))\bar{\otimes} 1\otimes\mathscr{L}(\mathbb{Z})$.
\end{proof}

\begin{Cor} \begin{enumerate}\item Up to isomorphism, there is only one irreducible $\lbrack (N,\Delta_N)\rbrack$-corepresentation of $\mathscr{L}^{\infty}(SU_q(2))$.
\item The von Neumann algebra $\widehat{N}$ (cf. Proposition \ref{PropDen2}) can be identified with $\oplus_{r\in \mathbb{Z}} B(l^2(\mathbb{N}))$.
\end{enumerate}
\end{Cor}

\begin{proof}
 The first statement follows immediately from the previous proposition, Theorem \ref{TheoPW}.2 and Proposition \ref{PropIndec}, as any $\mathcal{G}^{(n)}$ clearly implements the same irreducible $\lbrack(N,\Delta_N)\rbrack$-corepresentation. The second statement follows from Corollary \ref{CorFormN}.

\end{proof}

\begin{Prop}\label{PropFormV2} Denote $M=\mathscr{L}^{\infty}(SU_q(2))$. The elements $\frac{1}{q^t\sqrt{1-q^2}}\mathcal{G}_{ts}^{(n)}\xi_M$ form an orthonormal basis of $\mathscr{L}^2(N)$, and, under the identification $\mathscr{L}^2(N)\rightarrow \oplus_{n\in \mathbb{Z}} (l^2(\mathbb{N})\otimes l^2(\mathbb{N}))$ by sending $\frac{1}{q^t\sqrt{1-q^2}}\mathcal{G}_{ts}^{(n)}\xi_M$ to $e_{n,t}\otimes e_{n,s}$, the element \[\sum_{n\in \mathbb{Z}}\sum_{i,j=0}^{\infty} 1\otimes e_{n,ts}\otimes \mathcal{G}^{(n)}_{ts} \in \widehat{N}'\bar{\otimes}N\] equals $\widetilde{V}$.\end{Prop}

\begin{proof}
\noindent Remark first that $\varphi_{+}((\mathcal{G}^{(n)}_{ts})^*\mathcal{G}^{(n)}_{ru}) = \varphi_{+}(\mathcal{G}_{ts}^*\mathcal{G}_{ru})$. But as we have that \[\alpha(e_{ij}) = \sum_{k,l\in\mathbb{N}} \mathcal{G}_{ik}^*\mathcal{G}_{jl}\otimes e_{kl}\] and $(\varphi_+\otimes \iota)\alpha = \phi_{\alpha}$, it follows immediately that $\varphi_+(\mathcal{G}_{ik}^*\mathcal{G}_{jl}) = \delta_{kl}\delta_{ij} (1-q^2)q^{2i}$. Combined with the previous proposition, this proves that the $\frac{1}{q^t\sqrt{1-q^2}}\mathcal{G}_{ts}^{(n)}\xi_M$ form an orthonormal basis.\\

\noindent Now \[\widetilde{V}\, \mathcal{G}_{ij}^{(n)}\xi_M\otimes \xi_M = \sum_{k\in \mathbb{N}}\, \mathcal{G}_{ik}^{(n)}\xi_M\otimes \mathcal{G}_{kj}^{(n)}\xi_M.\] On the other hand, denoting $\widetilde{V}'=\sum_{n\in \mathbb{Z}}\sum_{i,j=0}^{\infty} 1\otimes e_{n,ts}\otimes \mathcal{G}^{(n)}_{ts}$, we have \begin{eqnarray*} \widetilde{V}'\,\mathcal{G}_{ij}^{(n)}\xi_M\otimes \xi_M  &\cong& q^i\sqrt{1-q^2} \sum_{k\in \mathbb{N}} e_{n,i}\otimes e_{n,k}\otimes \mathcal{G}_{kj}^{(n)}\xi_M\\ &\cong & \sum_{k\in \mathbb{N}} \mathcal{G}_{ik}^{(n)}\xi_M\otimes \mathcal{G}_{kj}^{(n)}\xi_M.\end{eqnarray*} As $\xi_M$ is separating for $N$, this shows that $\widetilde{V}=\widetilde{V}'$.

\end{proof}

\noindent One may deduce from this that the ordinary matrix units $e_{n,ts}$ in $\widehat{N} = \oplus_{n\in \mathbb{N}} B(l^2(\mathbb{N}))$ are of the form we required in the first section, namely their corresponding vectors $e_{n,t} \in l^2(\mathbb{N})$ are eigenvectors for the trace class operator $T$ implementing $\phi_{\alpha}$ (it should be remarked that we are implicitly using Proposition \ref{PropComp2} here).\\

\noindent We now remark that in \cite{DeC2}, we had already computed that the reflection of $\mathscr{L}^{\infty}(SU_q(2))$ across $(N,\Delta_N)$ (see section 4 for the terminology) may be identified with Woronowicz' quantum group $\mathscr{L}^{\infty}(\widetilde{E}_q(2))$ (see \cite{Wor3}), which has as its associated von Neumann algebra \[\mathscr{L}^{\infty}(\widetilde{E}_q(2)) = B(l^2(\mathbb{Z}))\bar{\otimes}1\bar{\otimes} \mathscr{L}(\mathbb{Z}) \subseteq B(\mathscr{L}^2(N)).\] Now it is known (see \cite{Baa2}) that this is a unimodular quantum group, with its invariant nsf weight $\varphi_0$ determined by \[\varphi_0(e_{ij}\otimes S^k) = \delta_{ij} \delta_{k,0} q^{2i}.\]

\begin{Prop} The modular element $\delta_{\widehat{N}}$ (see Theorem \ref{TheoGalob}.2) equals $\oplus_{n\in \mathbb{Z}} q^{2n}T^2 \in \widehat{N}$, where $T\in B(l^2(\mathbb{N}))$ is the operator $Te_i = q^{2i}e_i$.\end{Prop}

\begin{proof} Denote again by $A_n$ the $n$-th component of $\delta_{\widehat{N}}$ in $\widehat{N}$. Then by Proposition \ref{PropFin}.3, we know that $A_n = d_n^2 T^2$ for some $d_n>0$. Moreover, by Proposition \ref{PropFormV2} and Theorem \ref{TheoFormW}, we know that $\varphi_0(\mathcal{G}_{00}^{(n)}(\mathcal{G}_{00}^{(n)})^*) = \frac{1}{(1-q^2)d_n^2}$. So to know $d_n$, we should compute $\varphi_0(v^n L_{0+}L_{0+}^* (v^*)^n)$. But as $\sigma_t^{\varphi_0}(v) = q^{-2it}v$, we have that $\varphi_0(v^n L_{0+}L_{0+}^* (v^*)^n) = q^{-2n} \varphi_0(L_{0+}L_{0+}^*)$. As the $d_n$ are only determined up to a fixed scalar multiple anyway, we see that we may take $d_n = q^{n}$, which ends the proof.
\end{proof}

\begin{Prop} For all $m,n\in \mathbb{Z}$ and $i,j,k,l\in \mathbb{N}$, we have \[\varphi_0(\mathcal{G}_{ij}^{(n)}(\mathcal{G}_{kl}^{(m)})^*) = \delta_{mn}\delta_{ik}\delta_{jl} \frac{1}{q^{2n+2j}}.\]\end{Prop}

\begin{proof} By Proposition \ref{TheoFormW} and the previous proposition, we have that \[\varphi_0(\mathcal{G}_{ij}^{(n)}(\mathcal{G}_{kl}^{(m)})^*) = \delta_{mn}\delta_{ik}\delta_{jl} \frac{c}{q^{2n+2j}}\] for a certain constant $c$. This constant is precisely the number $\varphi_0(L_{0+}L_{0+}^*)$ which we neglected to compute in the previous proposition. But \begin{eqnarray*} \varphi_0(L_{0+}L_{0+}) &=& \sum_{k\in \mathbb{Z}} q^{2k} (q^{2k+2};q^2)_{\infty} \\ &=& (q^2;q^2)_{\infty}\sum_{k\in \mathbb{N}} \frac{q^{2k}}{(q^2;q^2)_k} \\ &=& 1,\end{eqnarray*} by the q-binomial theorem.
\end{proof}

\noindent \emph{Remark:} These orthogonality relations can also be written out in terms of the Wall polynomials $p_n$. Then one would simply get back the well-known orthogonality relations between these polynomials (see e.g. \cite{Koo1}, equation (2.6)), from which the above orthogonality relations can also be directly deduced.\\

\noindent As a final computation, let us determine the fusion rules between the $\mathcal{G}^{(n)}$ and the irreducible corepresentations $U_r$ of $\mathscr{L}^{\infty}(SU_q(2))$ (where $r\in \frac{1}{2}\mathbb{N}$).

\begin{Prop} For all $n\in \mathbb{Z}$ and $r\in \frac{1}{2}\mathbb{N}$, we have \[\mathcal{G}^{(n)}\times U_r \cong \oplus_{i=0}^{2r} \mathcal{G}^{(n-2r+2i)}.\]\end{Prop}

\begin{proof} By multiplying to the left with powers of $v$, it is easy to see that the fusion rules are invariant under translation of $n$. We may therefore restrict to the case $n=0$.\\

\noindent It is then well-known that if we consider $\mathscr{L}^{\infty}(S_{q0}^2)$ as an $\mathscr{L}^{\infty}(SU_q(2))$-comodule, the coaction $\alpha$ splits as $\oplus_{s\in \frac{1}{2}\mathbb{N}} U_{2s}$. From the proof of Proposition \ref{PropFus} and the remark preceding it, we obtain that $\mathcal{G}\times U_r$ will then split as a direct sum of less than $2r+1$ corepresentations. By the orthogonality relations between the $\mathcal{G}^{(n)}_{ij}$, it then suffices to find in each $\mathcal{G}^{(2i-2r)}$, with $0\leq i\leq 2r$, a component which has non-trivial scalar product with a matrix element of $\mathcal{G}\times U_r$.\\

\noindent By looking at the border of $U_r$, and by using $G_{00}^{(0)} = L_{0+}$, we find in $\mathcal{G}\times U_r$ the elements $L_{0+}a^{i}b^{2r-i}$ (up to a non-zero scalar), where $i\in \mathbb{N}$ with $i\leq 2r$. It is then enough to find inside $\mathcal{G}^{(2i-2r)}$ some element which has non-trivial scalar product with $L_{0+} a^{i}b^{2r-i}$. But by an easy computation, we have $L_{0+}a = vL_{0+}(1-b^*b)$, and then by induction $L_{0+}a^ib^{2r-i} = v^iL_{0+}(q^{-2i}b^*b;q^2)_{i}b^{2r-i}$. On the other hand, $\mathcal{G}^{(2i-2r)}_{0,2r-i}$ equals $v^iL_{0+}b^{2r-i}$, up to a scalar. As $ (q^{-2i}b^*b;q^2)_{i}(b^*b)^{2r-i}L_{0+}^*L_{0+}$ is a non-zero positive operator, we find that indeed $\varphi_+((\mathcal{G}^{(2i-2r)}_{0,2r-i})^* L_{0+}a^{i}b^{2r-i})\neq 0$, which then finishes the proof.

\end{proof}
\newpage
\noindent \emph{Remarks:} \begin{enumerate} \item By the discussion following Proposition \ref{PropFus}, we could also have deduced these fusion rules directly from \cite{Tom1}, as the multiplicity diagram of the ergodic coaction of $SU_q(2)$ on the standard Podle\`s sphere is explictly computed there.
\item In \cite{DeC5}, we discussed the `reflection technique' (cf.~ section 4) with respect to another action of $SU_q(2)$ on a type $I$ factor, namely the von Neumann algebraic completion of its action on the `quantum projective plane' (see e.g.~ \cite{Haj1}). We showed that the reflected quantum group in this case was the extended $\widetilde{SU}_q(1,1)$ quantum group of Koelink and Kustermans (\cite{Koe3}). This shows in particular that the Galois co-object $(\widetilde{N},\Delta_{\widetilde{N}})$ constructed from the quantum projective plane action is different from the one we considered in this section. In fact, as the multiplicity diagram of this action was explicitly computed in \cite{Tom1}, we see that the $(\widetilde{N},\Delta_{\widetilde{N}})$-projective representations of $SU_q(2)$ are labeled by the forked half-line $\{0_+,0_-\}\cup \mathbb{N}_0$ (again by the discussion following Proposition \ref{PropFus}). Now it can be shown that the quantum group $\widetilde{SU}_q(1,1)$ contains only two group-like unitaries. By Proposition 3.5 of \cite{DeC1}, this implies that the associated $\lbrack (\widetilde{N},\Delta_{\widetilde{N}})\rbrack$-projective corepresentations still form a (countably) infinite family. We do not know if this family of `ergodic actions on type $I$ factors' already occurs somewhere in the literature. At the moment, we have not even succeeded in explicitly describing these ergodic coactions, except for two cases. One case is the action on the quantum projective plane itself, which we will denote as $\alpha$. The other coaction $\widetilde{\alpha}$ is obtained by amplifying with the spin 1/2-representation $U_{1/2}$ of $SU_q(2)$: \[\widetilde{\alpha}: B(\mathscr{H}_{\alpha})\otimes M_2(\mathbb{C})\rightarrow M\bar{\otimes} B(\mathscr{H}_{\alpha})\otimes M_2(\mathbb{C}): x\rightarrow (U_{1/2})_{13}^* (\alpha\otimes \iota)(x)(U_{1/2})_{13}.\] Note that this is still irreducible, as the spectral decomposition of $\alpha$ only contains corepresentations corresponding to even integer spin by \cite{Tom1} (see the proof of Proposition \ref{PropFus}). Remark that the ergodic coaction $\widetilde{\alpha}$ is not isomorphic to a coideal of $SU_q(2)$ (again by \cite{Tom1}).

\end{enumerate}

\end{document}